\documentclass [12pt,a4paper,reqno]{amsart}
 \input amssymb.sty
\usepackage{mathabx}

\input xy
\xyoption{all}

\def\chv{\widecheck{v}}
\def\chz{\widecheck{z}}
\def\chmfa{\widecheck{\mfa}}

\def\htM{\widehat M}
\def\htU{\widehat U}
\def\htx{\hat x}
\def\htv{\hat v}

\def\bvrp{\overline \vrp}
\def\vrp{\varphi}
\def\m{m}

\newcommand{\pusho}[1]{{#1}_{*}}
\newcommand{\tpusho}[1]{{#1}_{* , \tng}}

\newcommand{\ds}[1]{\ {#1} \ }

\def\iso{\tilde{\to}}

\def\sm{\setminus}

\def\AEi{\operatorname{AE1}}
\def\AEii{\operatorname{AE2}}
\def\AEiii{\operatorname{AE3}}

\def\tCov{\Cov_\tng}

\def\tsCov{\Cov_{\tng,\stg}}

\def\tcaCov{\Cov_{\tng, c,  \ga}}

\def\tcpCov{\Cov_{\tng, c,  \gp}}

\def\tng{{\operatorname{t}}}
\def\stg{{\operatorname{s}}}

\newcommand{\etype}[1]{\renewcommand{\labelenumi}{(#1{enumi})}}

\def\gq{\mathfrak q}
\def\gp{\mathfrak p}
\def\go{\mathfrak o}
\def\ga{\mathfrak a}

\def\tT{\mathcal T}

\def\tG{\mathcal G}
\def\tE{E_\tng}

\def\taE{E_{\tng} (\ga) }

\def\MFCE{{MFCE}}

\def\tac{c_{\tng,\ga}}

\def\tpc{c_{\tng,\gp}}

\def\zset{\{ 0 \}}
\def\smin{\setminus}
\def\pogm{\pusho{\gm}}
\def\tpogm{\tpusho{\gm}}

\def\tapi{\pi_{\tng,\ga}}

\def\tppi{\pi_{\tng,\gp}}

\def\taUpi{\pi_{\tng,\ga,U}}

\def\zz{Z}

\def\endbox{ \hfill\quad\qed}

\def\eroman{\etype{\roman}}
\def\ealph{\etype{\alph}}

\def\pSkip{\vskip 1.5mm \noindent}

\def\sat{\operatorname{sat}}

\def\N{\mathbb N}

\def\mfD{\mathfrak D}

\def\mfA{\mathfrak A}

\def\mfq{\mathfrak q}
\def\mfp{\mathfrak p}
\def\mfa{\mathfrak a}
\def\mfb{\mathfrak b}
\def\mfc{\mathfrak c}

\def\olM{\overline M}
\def\olU{\overline U}
\def\olE{\overline E}
\def\olF{\overline F}

\def\al{\alpha}
\def\bt{\beta}
\def\gm{\gamma}
\def\Gm{\Gamma}

\def\mm{\frak m}
\def\tv{\tilde{v}}

\def\ga{\frak a}

\newtheorem{thm}{Theorem} [section]
\newtheorem*{thm*}{Theorem}
\newtheorem{cor}[thm]{Corollary}
\newtheorem{lem}[thm]{Lemma}
\newtheorem{lemma}[thm]{Lemma}
\newtheorem{prop}[thm]{Proposition}

\newtheorem*{claim*} {Claim}
\newtheorem*{theorem6.6'} {Theorem 6.6$'$}
\newtheorem{acknowledgment*}[thm] {Acknowledgment}

\newtheorem{ques}[thm]{Questions}

\newtheorem{example}[thm]{Example}
\newtheorem{examp}[thm]{Example}

\newtheorem{examples}[thm]{Examples}
 \newtheorem{rem}[thm]{Remark}
 \newtheorem{remark}[thm]{Remark}
  \newtheorem{remarks}[thm]{Remarks}
 \newtheorem*{remark*}{Remark}
 \newtheorem{defn}[thm]{Definition}

\newtheorem{schol}[thm]{Scholium}
\newtheorem{notation}[thm]{Notation}

\newtheorem{problem}[thm]{Problem}
\newtheorem*{notation*} {Notation}
\newtheorem*{notations*} {Notations}
\newtheorem*{comment*} {Comment}

\newcommand{\thmref}[1]{Theorem~\ref{#1}}
\newcommand{\propref}[1]{Proposition~\ref{#1}}

\newcommand{\lemref}[1]{Lemma~\ref{#1}}

 \renewcommand{\sectionmark}[1]{}

\newcommand{\bfem}[1]{\textbf{#1}}

\newcommand{\diag}{\operatorname{diag}}

\newcommand{\Cov}{\operatorname{Cov}}

 \newcommand{\dl}{\delta}
\newcommand{\Dl}{\Delta}
\newcommand{\lm}{\lambda}

 \newcommand{\id}{\operatorname{id}}

 \newcommand{\supp} {\operatorname{supp}}

\begin{document}

\title[Dominance and transmissions in supertropical valuation theory] {Dominance and transmissions \\\vskip 2mm in supertropical valuation theory}
\author[Z. Izhakian]{Zur Izhakian}
\address{Department of Mathematics, Bar-Ilan University, 52900 Ramat-Gan,
Israel}
\email{zzur@math.biu.ac.il}
\author[M. Knebusch]{Manfred Knebusch}
\address{Department of Mathematics,
NWF-I Mathematik, Universit\"at Regensburg 93040 Regensburg,
Germany} \email{manfred.knebusch@mathematik.uni-regensburg.de}
\author[L. Rowen]{Louis Rowen}
 \address{Department of Mathematics,
 Bar-Ilan University, 52900 Ramat-Gan, Israel}
 \email{rowen@macs.biu.ac.il}

\thanks{The research of the first and third authors have  been  supported  by the
Israel Science Foundation (grant No.  448/09).}

\thanks{The research of the second author was supported in part by
 the Gelbart Institute at
Bar-Ilan University, the Minerva Foundation at Tel-Aviv
University, the Department of Mathematics   of Bar-Ilan
University, and the Emmy Noether Institute at Bar-Ilan
University.}

\thanks{The research of the first author was supported  by the
Oberwolfach Leibniz Fellows Programme (OWLF), Mathematisches
Forschungsinstitut Oberwolfach, Germany.}


\date{\today}




\begin{abstract} This paper is a sequel of \cite{IKR1}, where we
defined  \textbf{supervaluations} on a commutative ring $R$ and
studied a \textbf{dominance relation} $\vrp \geq \psi$ between
supervaluations $\vrp$ and $\psi$ on $R$, aiming at an enrichment
of the algebraic tool box for use in tropical geometry.

A supervaluation $\vrp: R \to U$ is a multiplicative map from $R$
to a supertropical semiring $U$, cf.
\cite{IzhakianRowen2007SuperTropical},
\cite{IzhakianRowen2008Matrices}, \cite{IKR1}, with further
properties, which mean that $\vrp$ is a sort of refinement, or
covering, of an \m-valuation (= monoid valuation)  $v: R \to M$.
In the most important case, that $R$ is a ring, \m-valuations
constitute a mild generalization of valuations in the sense of
Bourbaki \cite{B}, while $\vrp \geq \psi$ means that $\psi: R \to
V$ is a sort of coarsening of the supervaluation $\vrp$. If
$\vrp(R)$ generates the semiring $U$, then $\vrp \geq \psi$ iff
there exists a ``\textbf{transmission}'' $\al:  U \to V$ with
$\psi = \al \circ \vrp$.

Transmissions are multiplicative maps with further properties, cf.
\cite[\S5]{IKR1}. Every semiring homomorphism $\al : U \to V$ is a
transmission, but there are others which lack additivity, and this
causes a major difficulty. In the main body of the paper we study
surjective transmissions via equivalence relations on
supertropical semirings, often much more complicated than
congruences by ideals in usual commutative algebra.
\end{abstract}

\maketitle

\tableofcontents

\baselineskip 14pt

\numberwithin{equation}{section}

\section*{Introduction}

We set forth a study in supertropical valuation theory begun in
\cite{IKR1}. Generalizing  Bourbaki's notion of a valuation on a
commutative ring \cite{B}, we there introduced
\emph{\m-valuations} (= monoid valuations) and then
\emph{supervaluations} on a commutative semiring. These are
certain maps from a semiring $R$ to a ``\emph{bipotent semiring}''
$M$ and a ``\emph{supertropical semiring}'', respectively.

To repeat, if $M$ is a \textbf{bipotent semiring}, here always
commutative, then the set $M$ is a totally ordered monoid  under
multiplication  with smallest element $0$, and the addition is
given by $x+ y = \max(x,y)$. Then an \textbf{\m-valuation} on $R$
is a multiplicative map $v: R \to M $, which sends $0$  to $0$,
$1$ to $1$, and obeys the rule $v(a+b) \leq v(a) + v(b).$ We call
$v$ a \textbf{valuation}, if moreover the semiring $M$ is
cancellative. \{In the classical case of a Krull valuation $v$,
$R$ is a field and $M = \tG \cup \{0\}$, with $\tG$ the valuation
group of $v$ in multiplicative notation.\}

A \textbf{supertropical semiring} $U$ is a -- here always
commutative -- semiring such that $e:= 1+1$ is an idempotent of
$U$ and some axioms hold (\cite[\S3]{IKR1}), which imply in
particular that the ideal $M := eU$ is a bipotent semiring. The
elements of $M \sm \{ 0 \}$ are called \textbf{ghost} and those of
$\tT(U) := U \sm M$ are called \textbf{tangible}. The zero element
of $U$ is regarded both ghost and tangible. For $x \in U$ we call
$ex$ the \textbf{ghost companion} of $x$. For $x,y \in U$ we have
the rule
$$ x + y = \left\{
\begin{array}{llll}
  y  &  & \text{if} &  ex < ey, \\
  x  &  & \text{if} &  ex > ey, \\
  ex  &  & \text{if} &  ex = ey. \\
\end{array}
\right.$$ Thus the addition in $U$ is uniquely defined by the
multiplication and the element $e$.  We also mention that $ex = 0
$ implies $x = 0$. We refer to \cite[\S3]{IKR1} for all details.

Finally, a \textbf{supervaluation} on $R$ is a multiplicative map
$\vrp: R \to U$ to a supertropical semiring $U$ sending $0$ to $0$
and $1$ to $1$, such that the map $e \vrp : R \to eU$, $a \mapsto
e \vrp(a)$, is an \m-valuation. We then say that $\vrp$
\textbf{covers} the \m-valuation  $v:= e\vrp$.

If $\vrp: R \to U$ is a  supervaluation  then $U' := \vrp(R) \cup
e \vrp(R)$ is a sub-semiring of $U$ and is again supertropical. In
practice we nearly always may replace $U$ by $U'$ and then have a
supervaluation at hands which we call \textbf{surjective}.

Given a surjective supervaluation $\vrp: R \to U$ and a map $\al:U
\to V$ to a supertropical semiring $V$, the map $\al \circ \vrp$
is again a supervaluation iff $\al$ is multiplicative, sends $0$
to $0$, $1$ to $1$, $e$ to $e$, and restricts to a semiring
homomorphism from $eU$ to $eV$. \{We denote the elements $1+1$ in
$U$ and $V$ both by ``$e$''.\} We call such a map $\al: U \to V$ a
\textbf{transmission}. Any semiring homomorphism from $U$ to $V$
is a transmission, but usually there exist also many transmissions
which are not additive.

The study of transmissions is the central topic of the present
paper. Transmissions are tied up with the relation of
\textbf{dominance} defined in \cite[\S5]{IKR1}. If $\vrp: R \to U
$ and $\psi: R \to V$ are supervaluations and $\vrp$ is
surjective, then $\vrp$ \textbf{dominates} $\psi$, which we denote
by $\vrp \geq \psi$, iff there exists a transmission $\al: U \to
V$ with $\psi =\al \circ \vrp$.

Already in \cite{IKR1} we studied dominance for supervaluations
which cover a fixed, say,  surjective \m-valuation  $v: R \to M$.
We called two such supervaluations $\vrp, \psi$
\textbf{equivalent} if $\vrp \geq \psi$ and $\psi \geq \vrp$. The
set $\Cov(v)$ of equivalent classes   $[\vrp]$ of supervaluations
$\vrp: R \to U$ covering $v$ (having varying target $U$ with $eU
=M$) turns out to be a complete lattice under the dominance
relation \cite[\S7]{IKR1}.

The bottom element of $\Cov(v)$ is the class $[v]$, with $v$
viewed as a supervaluation. The top element is given by a
surjective supervaluation $\vrp_v:R \to U(v)$, which we could
describe explicitly in the case that $v$ is valuation, i.e., $M$
is cancellative \cite[Example 4.5 and Corollary 5.14]{IKR1}.

We come to the contents of the present paper.  If $v: R \to M$ is
an \m-valuation and $\gm: M \to N$ is a homomorphism from $M$ to a
bipotent semiring $N$, then $\gm \circ v$ clearly  again is an
\m-valuation, called a \textbf{coarsening} of $v$. This
generalizes the usual notion of coarsening for Krull valuations.
It is of interest to look for relations between the lattices
$\Cov(v)$ and $\Cov(\gm \circ v)$. \S1 gives a first step in this
direction. Given $\gm: M \to N$ and a supertropical semiring $U$
with ghost ideal $M$ we look for transmissions $\al : U \to V$
which \textbf{cover} $\gm$, i.e., $V$ has the ghost ideal $N$ and
$\al(x) = \gm (x)$ for $x \in M$. Assuming that $\gm$ is
surjective, we prove that there exists an \textbf{initial} such
transmission  $\al = \al_{U,\gm}: U \to U_\gm$. This means that
any other transmission $\al': U \to V'$ covering $\gm$ is obtained
from $\al$ by composition with a transmission $\bt: U_\gm \to V'$
covering the identity of~$N$. This allows us to define an order
preserving  map $$\gm_*: \Cov(v) \to \Cov(\gm \circ v),$$ sending
a supervaluation $\vrp: R \to U$ to $\gm_*(\vrp) : = \al_{U,\gm}
\circ \vrp$. In good cases $\al_{U,\gm}$ has a ``pushout property"
(cf. Definition \ref{defn1.2}), that is even stronger than to be
initial, and $\al_{U,\gm}$ can be described explicitly (cf.
Theorem \ref{thm1.12}).

We defined in \cite[\S2]{IKR1} \textbf{strong valuations} and  in
\cite[\S9]{IKR1} \textbf{strong supervaluations}, which by
definition are covers of strong valuations. Tangible strong
supervaluations seems to be the most suitable  supervaluations for
applications in tropical geometry, hence our interest in them.
Given a strong strong supervaluation $v: R\to M$ we proved that
the set $\tsCov(v)$ of tangible strong supervaluations is a
complete sublattice of $\Cov(v)$ \cite[\S10]{IKR1}. In particular
this set is not empty.  In \S2 of the present paper we study the
behavior of such supervaluations covering $v$ under the map
$\gm_*$ from above. It turns out that $\gm_*(\tsCov(v)) \subset
\tsCov(\gm \circ v)$.

Denoting a representative of the top element of $\tsCov(v)$ by
$\bvrp_v$, we observe that $\gm_*([\bvrp_v])$ is most often
different from $[\bvrp_{\gm \circ v}]$. On the other hand,
$\gm_*([\vrp_v]) = [\vrp_{\gm \circ v}]$. This indicates that it
is not advisable to restrict supervaluation theory from start to
strong supervaluations, even if we are only interested in these.

The rest of the paper is devoted to an analysis and examples of
surjective transmissions. After a preparatory \S3, in which the
construction of a large class of supertropical semirings is
displayed, we study in \S4 \emph{``transmissive'' equivalence
relations.}

We call an equivalence relation  $E$ on a supertropical semiring
$U$ \textbf{transmissive}, if $E$ is multiplicative (= compatible
with multiplication), and the set of $E$-equivalence classes $U/
E$ admits the structure of a supertropical semiring such that the
natural map $\pi_E: U \to U/E$ is a transmission. (There can be at
most one such semiring structure on the set $U/E$.) Every
surjective transmission $\al: U\to V$ has the form $\al \circ
\pi_E$ with a (unique) transmissive equivalence relation $E$ and
an isomorphism $\rho: U/E \iso V$. Thus having a hold on the
transmissive equivalence relations means understanding
transmissions in general.

In all following $U$ denotes a supertropical semiring. The main
result of \S4 is an axiomatic description of those transmissive
equivalence relations $E$ on $U$, for which the ghost ideal of
$U/E$ is a cancellative semiring (Theorem \ref{thm4.7}, Definition
\ref{defn4.5}). We also give a criterion  that the transmission
$\pi_E$ is pushout, as defined in \S1 (Theorem \ref{thm4.13}), and
we analyse, which  ``orbital'' equivalence relations, defined in
\cite[\S8]{IKR1}, are transmissive. These exhaust \emph{all}
transmissive equivalence relations on $U$, if $U$ is a
\emph{supertropical semifield}, i.e., all tangibles $\neq 0 $ are
invertible in $U$, and all ghosts $\neq 0$ are invertible in $eU$.

We call a transmissive equivalence relation on $U$
\textbf{homomorphic} if the map $\pi_E : U \to U/E$ is a semiring
homomorphism. In \S5 we discuss a very special and easy, but
important class of such equivalence relations. Then in the final
section \S6 we look at homomorphic  equivalence relations in
general.

Given a homomorphic  equivalence relations $\Phi$ on $M:= eU$ we
classify all homomorphic  equivalence relations $E$ on $U$ which
extend $\Phi$. Here  \textbf{additivity} of $E$, i.e.,
compatibility with addition, causes the main difficulty. Thus, to
ease understanding, we first perform the classification program
for additive equivalence relations (Theorem $6.6'$), and then add
considerations on multiplicativity  to find the homomorphic
equivalence relations (Theorem~\ref{thm6.10}).

We close the paper with examples of homomorphic  equivalence
relations using the classification,  and also indicate
consequences for other transmissive  equivalence relations.

\begin{notations*}
Given sets $X,Y$ we mean by $Y \subset X$ that $Y$ is a subset of
$X$, with $Y  = X$ allowed. If $E$ is an equivalence relation on
$X$ then $X/E$ denotes the set of $E$-equivalence classes in $X$,
and $\pi_E: X \to X/E$ is the map which sends an element $x$ of
$X$ to its $E$-equivalence class, which we denote by $[x]_E$. If
$Y \subset X$, we put $Y/E := \{[x]_E  \ds | x \in Y\}.$

If $U$ is a supertropical semiring, we denote the sum $1+1$ in $U$
by $e$, more precisely by $e_U$ if necessary. If $x \in U$ the
\textbf{ghost companion} $ex$ is also denoted by $\nu(x)$ or
$x^\nu$, and the \textbf{ghost map} $U \to eU$, $x \mapsto
\nu(x)$, is denoted by $\nu_U$. If $\al: U \to V$ is a
transmission, then the semiring homomorphism $eU \to eV$ obtained
from $\al$ by restriction is denoted by  $\al^\nu$ and is called
the \textbf{ghost part} of $\al$. Thus $\al^\nu \circ \nu_U =
\nu_V \circ \al$.

If $v : R \to M $ is an \m-valuation we call the ideal $v^{-1}(0)$
of $R$ the \textbf{support} of $v$, and denote it by $\supp(v)$.
If $\vrp: R \to U$ is a supervaluation covering $v$, we most often
denote the equivalence  class $[\vrp] \in \Cov(v)$ abusively again
by $\vrp$
\end{notations*}

\section{Initial transmissions and a pushout property}\label{sec:1}

We state the main problem  which we address in this section.

\begin{problem}\label{prob1.1} Assume that $U$ is a supertropical
semiring with ghost ideal $eU=M,$ and $\gamma: M\to M'$ is a
semiring homomorphism from $M$ to a bipotent semiring $M'.$ Find a
supertropical semiring $U'$ with ghost ideal $eU'=M'$ and a
transmission $\alpha: U\to U'$ covering $\gamma,$ i.e.,
$\alpha^\nu=\gamma$ (cf. \cite[Definition 5.3]{IKR1}), with the
following universal property. Given a transmission $\beta: U\to V$
into a supertropical semiring $V$, with ghost ideal $N:=eV,$ and a
semiring homomorphism $\delta: M'\to N$, such that $\bt^\nu  = \dl
\gm$, there exists a unique transmission $\eta: U'\to V$ such that
$\beta=\eta\circ\alpha$ and $\eta^\nu=\delta.$
\end{problem}

We indicate this problem by the following commuting diagram
$$
\begin{xy}
\xymatrix{   U \ar@/^/@<+1ex>[rr]^\beta \ar@{>}[r]_\alpha  & U'\ar@{.>}[r]_\eta & V  \\
 M \ar@{^{(}->}[u] \ar@{>}[r]_\gamma & M'
  \ar@{>}[r]_\delta \ar@{^{(}->}[u] &
 N \ar@{^{(}->}[u]
}
\end{xy}
$$
where the vertical arrows are inclusion mappings.

We call such a map $\alpha: U\to U'$ a \textit{pushout
transmission covering} $\gamma.$ This terminology alludes to the
fact that our universal property means that the left square in the
diagram above is a pushout (=cocartesian) square in the category
STROP, whose objects are the supertropical semirings, and whose
morphisms are the transmissions. To see this, just observe that a
map $\rho: L\to W$ from a bipotent semiring $L$ to a supertropical
semiring $W$ is transmissive iff $\rho$ is a semiring homomorphism
from $L$ to $eW$ followed by the inclusion~$eW\hookrightarrow W.$

It is now obvious that, for a given homomorphism $\gamma: M\to
M'$, Problem \ref{prob1.1} has at most one solution up to
isomorphism over $M'$ and $U.$ More precisely, if both $\alpha:
U\to U'$ and $\alpha_1: U\to U_1$ are solutions, there exists a
unique isomorphism $\rho: U'\to U_1$ of semirings over $M'$ with
$\alpha_1=\alpha'\circ\rho.$

We may cast the universal property above in terms of $\alpha$
alone and then arrive at the following formal definition.

\begin{defn}\label{defn1.2} We call a map $\alpha: U\to V$ between
supertropical semirings a \bfem{pushout transmission} if the
following holds:
\begin{enumerate}
\item[1)] $\alpha$ is a transmission.

\item[2)] If $\beta: U\to W$
is a transmission from $U$ to a supertropical semiring $W$ and
$\delta:eV\to eW$ is a  semiring homomorphism with
$\beta^\nu=\delta\circ\alpha ^\nu$, then there exists a unique
transmission $\eta: U\to W$ with $\eta^\nu=\delta$ and
$\beta=\eta\circ\alpha.$
\end{enumerate}
We then also say that $V$ is ``the" \bfem{pushout of} $U$
\bfem{along} $\gamma.$\end{defn}

The notion of a pushout transmission can be weakened by demanding
the universal property in Definition \ref{defn1.2} only for $W=V$
and $\delta$ the identity of $eV.$ This is still interesting.

\begin{defn}\label{defn1.3} We call a transmission $\alpha: U\to
V$ between supertropical semirings an \bfem{initial transmission},
if, for any transmission $\beta: U\to W$ with $eW=eV$ and
$\beta^\nu=\alpha^\nu$, there exists a unique semiring
homomorphism\footnote{Every transmission $\eta$ with $\eta^\nu$
injective is a homomorphism \cite[Proposition 5.10.iii]{IKR1}.}
$\eta: V\to W$ over $eV=eW$ with $\beta=\eta\circ \alpha.$
\end{defn}

Given a supertropical semiring $U$ and a semiring homomorphism
$\gamma: eU\to N$ with $N$ bipotent, it is again clear that there
exists at most one initial transmission $\alpha: U\to V$ covering
$\gamma$ (in particular, $eV=N)$ up to isomorphism over $U$ and
$N.$

We turn to the problem of existence, first for initial
transmissions and then for pushout transmissions. In the first
case we can apply results on supervaluations from \cite[\S4 and
\S7]{IKR1}, due to the following easy but important observation.

\begin{prop}\label{prop1.4} Let $\alpha: U\to V$ be a map between
supertropical semirings and $\gamma: eU\to eV$ a semiring
homomorphism. The following are equivalent:
\begin{enumerate}\item[a)] $\alpha$ is a transmission covering
$\gamma.$ \item[b)] $\alpha $ is a supervaluation on the semiring
$U$ with $\alpha(e_U)=e_V$ covering the strict \m-valuation $v : =
\gamma\circ \nu_U: U\to eV.$ \end{enumerate}\end{prop}

We then have the commuting diagram
$$\xymatrix{
           U %
            \ar[d]_{\nu_U}     \ar[rr]^{\al} \ar[rrd]^{v}   &&      V
       \ar[d]^{\nu_V}
       \\
      eU    \ar[rr]^{\gamma}   &&        eV  \ \ .
 }$$

\begin{proof} We have to compare the axioms SV1--SV4 in \cite[\S4]{IKR1} plus
the condition $\alpha(e)=e$ with the axioms TM1--TM5 in
\cite[\S5]{IKR1}. The axioms SV1--SV3   say literally the same as
TM1--TM3, and the condition $\alpha(e)=e$ is TM4.

We now assume that $\alpha$ fulfills TM1--TM4. For every $x\in U$
we have $\alpha(ex)=\alpha(e)\alpha(x)=e\alpha(x).$ That $\alpha$
is a transmission covering $\gamma$ means that
$\alpha(z)=\gamma(z)$ for all $z\in eU.$ This is equivalent to
$\alpha(ex)=\gamma(ex)$ for all $x\in U;$ hence to the condition
$e\alpha(x)=\gamma\circ\nu_U(x)$ for all $x\in U.$ But this means
that $\alpha$ is a supervaluation covering
$\gamma\circ\nu_U.$\end{proof}

\begin{thm}\label{thm1.5} Given a supertropical semiring $U$ with
ghost ideal $M:=eU$ and a surjective homomorphism $\gamma: M\to
M'$ to a bipotent semiring $M',$ there exists an initial
transmission $\alpha: U\to U'$ covering $\gamma.$\end{thm}

\begin{proof} We introduce the strict surjective valuation
$$v=\gamma\circ \nu_U: U\twoheadrightarrow M'.$$ By \cite[\S7]{IKR1}
there exists an initial surjective supervaluation $\varphi_v: U\to
U(v)$ covering $v.$ (In particular, $eU(v)=M'.) $ The other
surjective supervaluations $\psi: U\to V$ covering $\gamma$ are
the maps $\pi_T\circ\varphi_v$ with $T$ running through the set of
all MFCE-relations on $U(v),$ as explained in \cite[\S7]{IKR1}.

Let $f: =\varphi_v(e_U)$ and $e:=e_{U(v)}=1_M.$ \propref{prop1.4}
tells us that $\pi_T\circ\varphi_v$ is the initial transmission
covering $\gamma$ iff $f\sim_T e$ and moreover $T$ is finer than
any other MFCE-relation on $U(v)$ with this property. Now we
invoke the following easy lemma, to be proved below.

\begin{lemma}\label{lem1.6} If $W$ is a supertropical semiring and
$X$ is a subset of $W,$ there exists a unique finest MFCE-relation
$E$ on $W$ with $x\sim_E e_Wx$ for every $x\in X.$\end{lemma}

We apply the lemma to $W=U(v)$ and $X=\{f\}$, and obtain a finest
equivalence relation $T$ on $U(v)$ with $f\sim_T ef.$ But
$$ef=\nu_{U(v)}\circ \varphi_v(e_U)=v(e_U)=e.$$
Thus, $T$ is the unique finest MFCE-relation on $U(v)$ with
$f\sim_Te,$ and $T$ gives us the wanted initial transmission
$\alpha=\pi_T\circ\varphi_v.$\end{proof}

\begin{proof}[Proof of \lemref{lem1.6}] The set $\mathcal M$ of
all MFCE-relations $F$ on $W$ with $x\sim_F ex$ for all $x\in X$
is not empty, since it contains the relation $E(\nu_W).$ The
relation $E:=\bigwedge \mathcal M,$ i.e., the intersection of all
$F\in\mathcal M,$ has the desired property.
\end{proof}

\begin{notation}\label{notation1.7} We denote ``the" initial transmission in
\thmref{thm1.5} by $\alpha_{U,\gamma}$, the semiring~$U'$ by
$U_\gamma,$ and the equivalence relation $E(\alpha_{U,\gamma})$ by
$E(U,\gamma).$\end{notation}

This notation is sloppy, since $\alpha_{U,\gamma}$ is determined
by $U$ and $\gamma$ only up to isomorphism. But $E(U,\gamma)$
truly depends only on $U$ and $\gamma.$ The ambiguity for
$\alpha_{U,\gamma}$ can be avoided if $\gamma$ is surjective, due
to the following lemma.

\begin{lem}\label{lem1.8} If $\alpha: U\to V$ is an initial
transmission covering a surjective homomorphism  $\gamma: M\to
M'$, then $\alpha$ itself is a surjective map.\end{lem}

\begin{proof} $V_1:=\alpha(V)$ is a subsemiring of $V$ and thus a
supertropical semiring itself. Replacing $V$ by $V_1$ we obtain
from $\alpha$ a surjective transmission $\alpha_1: U\to V_1.$
Since $\alpha$ is initial there exists a unique transmission
$\eta: V\to V_1$ over $M'$ with $\alpha_1=\eta\alpha.$ Also
$\alpha=j\alpha_1$ with $j$ the inclusion from $V_1$ to $V.$ By
the universal property of $\alpha$ we conclude from
$\alpha=j\eta\alpha$ that $j\eta$ is the identity on $V.$ This
forces $V=V_1.$
\end{proof}

Thus, if $\gamma$ is surjective, we have a canonical choice for
$U_\gamma$ and $\alpha_{U,\gamma}$, namely,
$U_\gamma=U/E(U,\gamma)$ and
$\alpha_{U,\gamma}=\pi_{E(U,\gamma)}.$ Usually we will understand
by $U_\gamma$ and $\alpha_{U,\gamma}$ this semiring and
transmission.

In light of \thmref{thm1.5} our main Problem \ref{prob1.1} can be
posed as follows: Given $U$ and $\gamma,$ is $\alpha_{U,\gamma}:
U\to U_\gamma$ a pushout transmission?

\medskip

We assume in the following that $\gamma: M\to M'$ \textit{is
surjective and} $M'$ \textit{is a cancellative bipotent domain};
hence $v=\gamma\circ\nu_U$ is a strict surjective valuation. In
this case we will obtain a positive solution of the problem. The
point here is that we can give an explicit description of
$U_\gamma$ and $\alpha_{U,\gamma}$, which allows us to check the
pushout property.

We  already have an explicit description of $\varphi_v: U\to
U(v),$ given in \cite[\S4]{IKR1}. Thus all we need is an explicit
description of the finest MFCE-relation $T$ on $U(v)$ with
$f\sim_T e.$ We develop such a description in a more general
setting.

\medskip

Assume that $U$ is a supertropical semiring, $e:=e_U,$ and $f$ is
an idempotent of $U.$ The ideal $L:=fU$ of $U$ is again a
supertropical semiring with unit element~$f$ (under the addition
and multiplication of $U),$  since $L$ is a homomorphic image of
$U.$ We have $e_L=f+f=ef.$

If $F$ is an equivalence relation on the set $L,$ there is a
unique finest equivalence relation~$E$ on $U$ extending $F.$ It
can be described as follows. Let $x_1,x_2\in U.$ Then $x_1\sim_E
x_2$ iff either $x_1=x_2$ or $x_1\in L$, $x_2\in L$ and $x_1\sim_F
x_2.$ We call $E$ the \bfem{minimal extension} of the equivalence
relation $F$ to $U.$

\begin{lemma}\label{lem1.9} Let $F$ be an equivalence relation on
$fU$, and let $E$ denote the minimal extension of $F$ to $U.$

\begin{enumerate} \item[a)]
 If $F$ is multiplicative, then $E$ is multiplicative.

\item[b)]
 If $F$ is fiber conserving, so is $E.$
\end{enumerate}
\end{lemma}

\begin{proof} Assume that $x_1,x_2$ are elements of $U$ with
$x_1\sim_E x_2.$ Assume (without loss of generality) that also
$x_1\ne x_2.$ Then $x_1,x_2\in fU$ and $x_1\sim_F x_2.$

If $F$ is multiplicative then, for any $z\in U,$
$$x_1z=x_1(fz)\sim_F x_2(fz)=x_2z;$$
hence $x_1z\sim_E x_2z.$ Thus $E$ is multiplicative.

If $F$ is fiber conserving, then
$$ex_1=(ef)x_1=(ef)x_2=ex_2.$$ Thus $E$ is fiber conserving.
\end{proof}

\begin{prop}\label{prop1.10} Assume that $U$ is a supertropical
semiring, $e:=e_U,$ and $f$ is an idempotent of $U.$ We define a
binary relation $E$ on $U$ by decreeing $(x_1,x_2\in U)$
$$x_1\sim_E x_2\quad\text{iff either}\quad
x_1=x_2\quad\text{or}\quad x_1,x_2\in fU\quad\text{and}\quad
ex_1=ex_2.$$
\begin{enumerate}
\item[a)]
 $E$ is an MFCE-relation on $U.$

\item[b)] If $ef=e,$ then $e\sim_E f,$ and $E$ is finer than any other
multiplicative equivalence relation $E'$ on $U$ with
$e\sim_{E'}f.$
\end{enumerate}%
\end{prop}

\begin{proof} a) We apply the preceding lemma with $F$ the
relation $E(\nu_L)$ (cf. \cite[Example~6.4]{IKR1} on the
supertropical semiring $L:=fU.$ The minimal extension of $F$ to
$U$ is the relation $E$ defined in the proposition. Indeed, for
$x_1,x_2\in L$ we have $x_1\sim_F x_2$ if $efx_1=efx_2.$ Since
$fx_i=x_i$ $(i=1,2)$, this means that $ex_1=ex_2.$ By
\lemref{lem1.9} the relation $E$ is MFCE.

b) Assume now that $ef=e,$ i.e., $e\in L.$ Then $e\sim_E f$ by
definition of $E.$ Let $E'$ be any multiplicative equivalence
relation on $U$ with $e\sim_{E'} f.$ If $x_1,x_2\in U$ and
$x_1\sim_E x_2$ we want to conclude that $x_1\sim_{E'} x_2.$ We
may assume that $x_1\ne x_2.$ Then $x_1,x_2\in fU$ and
$ex_1=ex_2.$ Now $x_i\sim_{E'}ex_i$ $(i=1,2)$; hence $x_1\sim_{E'}
x_2,$ as desired.
\end{proof}

We are ready for a solution of Problem \ref{prob1.1} in the case
that $\gamma: M\to M'$ is surjective and $M'$ is a cancellative
bipotent semidomain; hence $v=\gamma\circ\nu_U$ is a strict
surjective valuation. As before, let $T$ denote the finest
MFCE-relation on $U(v)$ with $f\sim_Te$ for $e:=e_{U(v)}$ and
$f:=\varphi_v(e_U).$ Recall from the proof of \thmref{thm1.5} that
$ef=e.$ Thus \propref{prop1.10} applies. We spell out what the
proposition says in the present case.

For that we write the semiring $U(v)$ and the map $\varphi_v$ in a
way different from \cite[\S4]{IKR1}. Let $\htU$ denote a copy of
$U$ disjoint from $U$ with copying isomorphism $x\mapsto \htx.$ We
use this to distinguish an element $x\in U \setminus\mathfrak q,$
with $\mathfrak q:=\supp v,$ from the corresponding element in
$\mathcal T(U(v)).$ Thus we write $$
U(v)=(\htU\setminus\hat{\mathfrak q})\ \dot\cup\ M'$$ with
$\hat{\mathfrak q}:=\{\htx\bigm| x\in U,\, \gamma(e_Ux)=0\},$ and
$\varphi_v(x)=\htx$ for $x\in U\setminus\mathfrak q,$
$\varphi_v(x)=0$ for $x\in  \mathfrak q.$ Notice that $fU(v)=(\htM
\setminus{\hat{\mathfrak q}})\cup M'$ with $\htM:=\{\htx \ds |
x\in M\}.$

According to \propref{prop1.10} the equivalence relation $T$ has
the following description. Let $y_1,y_2\in U(v)$ be given with
$y_1\ne y_2.$ Then $y_1\sim_T y_2$ iff $y_1=\htx_1,$ $y_2=\htx_2$,
with either $x_1,x_2\in M$ and $\gamma(e_Ux_1)=\gamma(e_Ux_2)$ or
$x_1,x_2\in U$ and $\gamma(e_Ux_1)=\gamma(e_Ux_2)=0.$ We may
choose $U_\gamma=U(v)/T$ and
$\alpha_{U,\gamma}=\pi_T\circ\varphi_v.$ The transmission
$\alpha:=\alpha_{U,\gamma}$ is a surjective map from $U$
to~$U_\gamma$, and the equivalence relation $E(\alpha)$ is the
relation $E(U,\gamma)$ defined in Notation~\ref{notation1.7}. Thus
$E:=E(U,\gamma)$ has the following description: If $x_1,x_2\in U$
and $x_1\ne x_2$ then
$$x_1\sim_Ex_2\ \Leftrightarrow\ \gamma(e_Ux_1)=\gamma(e_Ux_2),\  \text{and
if}\ x_1\in \mathcal T(U)\ \text{or}\ x_2\in \mathcal T(U),\
\gamma(e_Ux_1)=0.$$ Having found $E(U,\gamma)$ we now redefine
$$U_\gamma : =U/E(U,\gamma), \qquad
\alpha_{U,\gamma}:=\pi_{E(U,\gamma)}.$$ We arrive at the following
theorem.

\begin{thm}\label{thm1.12}
Let $U$ be a supertropical semiring, $e:=e_U$, $M:=eU,$ and assume
that $\gamma:M \to M'$ is a surjective  homomorphism  from $M$ to
a cancellative bipotent semidomain $M'.$ Then $E:=E(U,\gamma)$ can
be described as follows $(x_1,x_2\in~U):$
$$
\begin{array}{lll}
x_1\sim_E x_2\ \ & \text{iff}\ & x_1=x_2,\ \\ & & \text{or}\ \
\gamma(ex_1)=\gamma(ex_2),\ ex_1=x_1,\ ex_2=x_2,\\ & & \text{or}\
\ \gamma(ex_1)=\gamma(ex_2)=0.\end{array}
$$
\end{thm}

\begin{schol}
 Thus this binary relation $E$ on $U$ is
a   multiplicative equivalence relation, and  the multiplicative
monoid $U/ E$  can be turned into a supertropical semiring in a
unique way such that $\pi_E: U\to U/E$ is a transmission. It is
the initial transmission covering $\gamma.$
\end{schol}

Most often $\pi_E$ is not a homomorphism, cf. \S\ref{sec:6} below.

\begin{thm}\label{thm1.13}
If $\gamma$ is surjective and $M'$ is a cancellative bipotent
semidomain, then $\alpha_{U,\gamma}$ is a pushout transmission.
\end{thm}

\begin{proof}
Let $\alpha:=\alpha_{U,\gamma}=\pi_E: U\to U/E$ with
$E:=E(U,\gamma).$ Assume that $\delta: M'\to N$ is a homomorphism
from~$M'$ to a bipotent semiring $N$ and $\beta: U\to V$ is a
transmission covering $\delta\gamma: M\to N,$ i.e., with
$e_V\beta=\delta\gamma e_U.$ (In particular $eV=N.)$

We want to verify that $\beta$ respects the equivalence relation
$E,$ i.e., given $x_1,x_2\in U,$ that
$$x_1\sim_Ex_2\quad\text{implies}\quad\beta(x_1)=\beta(x_2).$$

We may assume that $x_1\ne x_2.$ If $x_1$ or $x_2$ is tangible
then $\gamma(ex_1)=\gamma(ex_2)=0;$ hence
$e_V\beta(x_i)=\delta\gamma(ex_i)=0$ for $i=1,2.$ This implies
$\beta(x_1)=\beta(x_2)=0.$ Assume now that both $x_1$ and $x_2$
are ghost. Then $\gamma(ex_1)=\gamma(ex_2);$ hence
$\delta\gamma(ex_1)=\delta \gamma(ex_2),$ i.e.,
$e_V\beta(x_1)=e_V\beta(x_2).$ But both $\beta(x_1)$ and
$\beta(x_2)$ are ghost or zero. Thus $\beta(x_1)=\beta(x_2)$
again.

Since $\alpha$ is surjective, it follows that we have a
well-defined map $\rho: U/E\to V$ with $\beta=\rho\alpha.$ Now
\cite[Proposition 6.1.ii]{IKR1} tells us that $\rho$ is a
transmission, since both $\alpha$ and $\beta$ are transmissions
and $\alpha$ is surjective. We have
$$\nu_V\rho\alpha=\nu_V\beta = \delta \gamma \nu_U =\delta \nu_{U/E}\alpha.$$
Since $\alpha$ is surjective, this implies that $\nu_V\rho=\delta
\nu_{U/E},$ i.e., $\rho$ covers $\delta.$ The pushout property of
$\alpha$ is verified.\end{proof}

\begin{remark} If $\gm$ is surjective, but $M'$ is not assumed to
be cancellative, we have a description of $E(U,\gm)$ in
\cite[\S4]{IKRMon}, which is nearly as explicit as the description
above in Theorem \ref{thm1.12}, but then often $\al_{U,\gm}$ is
not a pushout transmission.
\end{remark}

Assume now that $U$ is \bfem{any} supertropical semiring, $M:=eU,$
and $\gamma: M\to M'$ is an \bfem{injective} semiring homomorphism
from $M$ to a bipotent semiring $M'.$ Then Problem \ref{prob1.1}
can be solved affirmatively in an easy direct way, as we explicate
 now.

We may assume, without loss of generality, that $M$ is a
subsemiring of $M'$ and $\gamma$ is the inclusion from $M$ to
$M'.$ We define a semiring $U'$ as follows. As a set, $U'$ is the
disjoint union of the sets $U$ and $M'\setminus M.$ We have
$U\subset U',$ $M'\subset U',$\ $U\cup M'=U',$ $U\cap M'=M.$
Let~$\nu$ denote the ghost map from $U$ to $M,$ $\nu=\nu_U.$ We
define  addition and multiplication on $U$ by taking the given
addition and multiplication on~$U$ and on~$M',$ and putting
$$x\cdot z=z\cdot x=\nu(x)\cdot z $$
$$x+
z=z+ x= \left\{
\begin{array}{lll}
  x  & \text{if} & \nu(x) > z\\[1mm]
  z  & \text{if} & \nu(x) \leq z\\
\end{array} \right.
$$ for $x\in U,$ $z\in M'.$ In the cases that
$x\in M$ and $z\in M',$ or $x\in U$ and $z\in M,$ these new
products are the same as the ones in $M'$ or $U,$ respectively.
Thus we have well-defined operations $\cdot$ and $+$ on $U'.$ One
checks in any easy and straightforward way that they obey all of
the  semiring axioms. Thus $U'$ is now a commutative semiring with
$1_{U'}=1_U.$ It clearly obeys the axioms (3.3$'$), (3.3$''$),
(3.3) in \cite{IKR1}. Thus $U'$ is supertropical. We have
$e_{U'}=e_U,$ $eU'=M',$ $\mathcal T(U')=\mathcal T (U).$

\begin{defn}\label{defn1.14}
We call $U'$ the supertropical  semiring \bfem{obtained from} $U$
\bfem{by extension of the ghost ideal} $M$ \bfem{to} $M'.$ We also
say, more briefly, that $U'$ is a \bfem{ghost extension of}~$U.$

\end{defn}

Let $\alpha$ denote the inclusion $U\hookrightarrow U'.$ It is
obvious that $\alpha$ is a transmission covering the inclusion
$\gamma: M\hookrightarrow M'.$ We verify that $\alpha$ is a
pushout transmission.

Let $\delta: M'\to N$ be a homomorphism from $M'$ to a bipotent
semiring $N$ and $\beta: U\to V$ a transmission covering
$\delta\gamma.$ This means that $eV=N,$ and
\begin{enumerate}
    \item $\beta(x) = \delta(x)$ for $x \in M.$
\end{enumerate}
 Clearly, we have a unique well-defined map
$\rho: U'\to V$ with $\rho|U=\beta$ and
\begin{enumerate}
    \item[(2)] $\rho(x) = \delta(x)$ for $x \in M'.$
\end{enumerate}
 We
have $\rho(0)=0,$ $\rho(1)=1,$ $\rho(e_{U'})=e_V.$ One checks
easily that $\rho$ is multiplicative.

We now know that $\rho$ is a transmission covering $\delta.$ We
have proved the following theorem.

\begin{thm}\label{thm1.15}
Assume that $M'$ is a bipotent semiring and $M$ is a subring of
$M'.$ Assume further that $U$ is a supertropical semiring with
ghost ideal $M,$ and $U'$ is the supertropical semiring obtained
from $U$ by extension of the ghost ideal $M$ to $M'.$ Then the
inclusion mapping $U\to U'$ is a pushout tranmission covering the
inclusion mapping $M\hookrightarrow M'.$\end{thm}

Combining Theorems \ref{thm1.13} and \ref{thm1.15}, we obtain the
most comprehensive solution of Problem~\ref{prob1.1} that we can
offer in this section.

\begin{thm}\label{thm1.16}\footnote{In \S5 and \cite[\S1]{IKRMon} we will meet pushout transmissions
which are not covered by this theorem.} Let $\gamma: M\to M'$ be a
homomorphism between bipotent semirings, and assume that the
bipotent semiring $\gamma(M)$ is cancellative. $\{$N.B. This holds
if $M'$ is cancellative.$\}$ Let $U$ be a supertropical semiring
with $eU=M.$ Then $\alpha_{U,\gamma}: U\to U_\gamma$ is a pushout
transmission.
\end{thm}

\begin{proof}
We have a factorization $\gamma=i\circ\bar\gamma,$ with
$\bar\gamma$ the map $x\mapsto \gamma(x)$ from $M$ to the
subsemiring $\gamma(M)$ of $M'$, and $i$ the inclusion from
$\gamma(M)$ to $M'.$ By Theorems \ref{thm1.13} and \ref{thm1.15}
there exist pushout transmissions $\alpha: U\to \overline U$ and
$\beta: \overline U\to U'$ covering $\bar \gamma$ and $i$,
respectively. Now look at the commutative diagram
$$
\begin{xy}
\xymatrix{   U   \ar@{>}[r]_\alpha  & \overline U\ar@{>}[r]_\beta & U'  \\
 M \ar  [u] \ar@{>>}[r]_ {\bar
 \gamma} & \gamma(M)
   \ar@{^{(}->}[r]_{i} \ar [u] &
  M' \ar  [u]
}
\end{xy}
$$
where the vertical arrows denote inclusions. Here the left and the
right square are pushout diagrams in the category STROP of
 supertropical semirings and transmissions. Thus also the outer
 rectangle is a pushout in this category (cf., e.g.,
 \cite[p.72, Execr.8]{ML}), i.e., $\beta\alpha$ is a pushout transmission.
 If $\alpha_{U,\gamma}: U\to U_\gamma$ is any prechosen initial
 covering of $\gamma,$  there exists an isomorphism $\rho:
 U'\to U_\gamma$ over $M'$ with
 $\rho\beta\alpha=\alpha_{U,\gamma}.$ Thus also $\alpha_{U,\gamma}$
 is a pushout transmission.
\end{proof}

\section{Pushouts of tangible supervaluations }\label{sec:2}

If $\vrp: R \to U$ and $\psi:R \to V$  are supervaluation on a
semiring $R$, and  $\vrp$  dominates $\psi$,  then we also say
that $\psi$ is a \textbf{coarsening} of $\vrp$ . Recall that this
happens iff there is a transmission $\al : U \to V$ with $\psi =
\al \circ \vrp $. If in addition $\vrp$ is surjective,   i.e., $U
= \vrp(R) \cup e \vrp(R)$, which is no essential loss of
generality, then $\al$ is uniquely determined by $\vrp$    and
$\psi$,  and we write $\al = \al_{\psi,\vrp}$ (cf.
\cite[\S5]{IKR1}).

Assume now that $v: R \to M$ is a surjective \m-valuation and
$\vrp: R \to U$ is a surjective supervaluation covering $v$ (in
particular $M = eU$). Moreover, let $\gm: M \to N$ be a surjective
homomorphism to another (bipotent) semiring $N$.

\begin{defn}\label{defn2.1} We say that a surjective
supervaluation  $\psi: R \to V$ is the \textbf{initial coarsening
of $\vrp$  along $\gm$}, if $\psi$ is a coarsening of $\vrp$ and
$\al_{ \psi, \vrp}$ is the initial transmission covering $\gm$
(cf. Definition \ref{defn1.3}). In the notation \ref{notation1.7};
which we will  obey in the following, this means that $V = U_\gm$
and $ \al_{\psi,\vrp} =  \al_{U, \gm}$.We then write $\psi =
\pusho{\gm}
 (\vrp)$.
\end{defn}

In this way we obtain a map $$ \pusho{\gm}  : \Cov(v) \to \Cov(\gm
v)$$ between complete lattices.

[We could define such a map $\pusho{\gm}$  also if $\gm : M \to N$
is not necessarily surjective. But in the present section this
will give no additional insight.]

\vskip 3mm \emph{In the following, we will tacitly assume that all
occurring supervaluations are surjective.}

We write down a functional property of the initial transmissions
$\al_{U,\gm} $, which will give us simple properties of the maps
$\pusho{\gm}$. The map $\gm : M \to N$  is always assumed to be a
surjective homomorphism between bipotent semirings (as before).

\begin{prop}\label{prop2.2}
Let $U$ and $V$ be supertropical semirings with $eU = e V = M$ and
let $\lm: U \to V$ be a transmission over $M$, hence a
homomorphism\footnote{Any transmission $U \to W$, which is
injective on $eU$, is a homomorphism, cf. \cite[Proposition
5.10.iii]{IKR1}.}.
\begin{enumerate} \ealph
    \item Then there exists a unique  transmission from $U_\gm $
    to $V_\gm$ over $N$, denoted by $ \lm_\gm$, such that
    $$ \lm_\gm \circ \al_{U,\gm} \   = \  \al_{V, \gm} \circ \lm.$$
We thus  have a commuting diagram
$$\xymatrix{
           V  %
                \ar[rr]^{\al_{V,\gm}}    &&
            V_\gm
       \\
      U  \ar[u]_{\lm}   \ar[rr]^{\al_{U,\gm}}   &&        U_\gm \ar[u]_{\lm_\gm}
       \\
      M  \ar@{^{(}->}[u]   \ar[rr]^{\gm}   &&        N \ar@{^{(}->}[u]
 }$$
with inclusion mappings $M \hookrightarrow U$ and $N
\hookrightarrow U_\gm$.

    \item If $ \xi :  V  \to W $ is a second homomorphism over $M$
    then
    $$ \xi_\gm \lm_\gm  \ = \ (\xi \lm )_\gm .$$
\end{enumerate}
\end{prop}

\begin{proof} a): $\al_{V,\gm} \lm : U \to V_\gm$ is a transmission
covering $\gm$. Now use the universal property of the initial
transmission $\al_{U,\gm}$.

b):  $\xi_\gm  \lm _\gm : U_\gm \to W_\gm$  is a transmission over
$N$ such that $$ \xi_\gm  \lm _\gm \al_{U,\gm} =   \xi_\gm
\al_{V,\gm} \lm = \al_{W,\gm} \xi \lm.$$ By the uniqueness  part
in a) we conclude that $ \xi_\gm  \lm _\gm  = (\xi \lm) _\gm$.
\end{proof}

As an immediate consequence of part b) we have

\begin{cor} \label{cor2.3}
The map $\pusho{\gm} : \Cov(v) \to \Cov (\gm v)$  is order
preserving in the weak sense, i.e., $\vrp \geq \psi$ implies
$\pusho{\gm}(\vrp) \geq \pusho{\gm}(\psi)$.
 \hfill\quad\qed

\end{cor}

\begin{cor} \label{cor2.4}
If $\vrp: R \to U$ and $\psi: R \to V $ are supervaluations
covering $v$ (in particular $eU = e V =M$) with $\vrp \geq \psi$
then
$$ \al_{\pusho{\gm}(\psi), \pusho{\gm} (\vrp)} \ = \ (\al_{\psi,\vrp})_\gm.$$
\end{cor}

\begin{proof}
We have $\psi = \lm \vrp$ with $\lm := \al_{\psi,\vrp}$.  From
this we conclude that
$$ \pusho{\gm}(\psi) \ = \ \al_{V,\gm} \lm \vrp \ = \
\lm_\gm \al _{U,\gm} \vrp \ = \ \lm_\gm \pusho{\gm}(\vrp).$$ Thus
$\lm_\gm$ is the transmission from $\pusho{\gm}(\vrp)$ to
$\pusho{\gm}(\psi)$.
\end{proof}

Starting from now \emph{we assume that the bipotent semirings $M$
and $N$ are cancellative}; hence $v: R \to M$ and $\gm v : R \to
N$ are valuations. We define
$$\gp  := \gm ^{-1}(0), \qquad \frak{q}
:= v^{-1} (0) = \supp(v), \qquad  \frak{q}' := v^{-1} (\gp) =
\supp(\gm v).$$
Notice that $\gp$, $\frak{q}$, $\frak{q}'$ are prime ideals of $M$
and $R$, respectively.

Given any supertropical semiring $U$ with $e U = M$, we now know
that $\al_{U, \gm}: U \to U_\gm$ is a pushout transmission
(Theorem \ref{thm1.13}). Consequently, if $\vrp \in \Cov(v)$, we
now call $\pusho{\gm} (\vrp)$ the \textbf{pushout of $\vrp$ along
$\gm$} (instead of ``initial coarsening of $\vrp$ along $\gm$'').

The good thing  is that we now have an explicit descriptions of
$U_\gm$ and $\al_{U,\gm}$ which we recall from Theorem
\ref{thm1.12}.

We start with a multiplicative equivalence relation $E(U, \gm)$ on
$U$ established in Theorem~\ref{thm1.12}.  To repeat, for $x,y$ in
$U$
$$
\begin{array}{lll}
   x \sim_{E(U,\gm)} y  &  \ \Longleftrightarrow \ &\text{either $x=y$}, \\
   &  & \text{or both  $x,y \in M$ and $\gm(x) = \gm(y)$},  \\
   &  & \text{or $ex \in \gp$, $ey \in \gp$}. \\
\end{array}
 $$

The restriction $E(U,\gm)| M$ is the equivalence relation $E(\gm)$
given by $\gm: M \twoheadrightarrow N$. We identify every  class
$[x]_{E(U,\gm)}$, $x \in M$, with the image $\gm(x) \in N$ and
then have
$$ M/ E(U,\gm) = N.$$

As proved in \S\ref{sec:1}, we may choose\footnote{Recall that
$\al_{U,\gm}: U
 \to U_\gm$ is the solution of a universal problem.} $U_\gm = U /
 E(U,\gm)$ and then have
$$ \al _{U,\gm } = \pi_{E(U,\gm)} : x \longmapsto [x]_{E(U,\gm)}. $$

Let $x \in \tT(U)$. If $ex \notin \gp$, then $ [x]_{E(U,\gm)} = \{
x \}$, but if $ex \in \gp $, then $ [x]_{E(U,\gm)} = 0 \in N$.
Thus we see that $\tT(U_\gm) = U_\gm \setminus  N$  is the
bijective image of $\{ x \in \tT(U) \ | \ ex \notin \gp \}$. We
identify $ [x]_{E(U,\gm)}$ with $x$, if $x$ lies in this set, and
then have
$$
\tT(U_\gm) = \{ x \in \tT(U) \ | \ ex \notin \gp \}, \qquad U_\gm
= \{ x \in \tT(U) \ | \ ex \notin  \gp \} \ \dot\cup \ N.$$

Notice that the multiplicative monoid $\tT(U_\gm)$  has become a
submonoid of $\tT(U)$, since  $E(U,\gm)$ is multiplicative, but
the sum of two elements of $\tT(U_\gm)$, computed in the semiring
$U_\gm$, can be very different from their sum in $U$.

After all these identifications we have
\begin{lem}\label{lem2.5} For any $x \in U,$
$$ \al_{U,\gm}(x) \ = \ \left\{
\begin{array}{ll}
  x & \text{if } x \in \tT(U), ex \notin \gp,   \\[2mm]
  0 & \text{if } x \in \tT(U), ex \in \gp,  \\[2mm]
  {\gm}(x)  & \text{if } x \in M .\\
\end{array}
\right.
$$ \hfill\quad\qed
\end{lem}

Given a surjective valuation $v: R \to M$, as before, we denote by
$\tCov(v)$ the set of all (equivalence classes of) tangible
supervaluations covering $v$. It is an upper set, and hence a
complete sublattice of the lattice $\Cov(v)$ with the same top
element $\vrp_v: R \to U(v)$ as $\Cov(v)$ (cf. \cite[\S10]{IKR1}).

Let $D(M)$ denote the unique supertropical semiring $U$ such that
$eU= M$ and $\nu_U$ maps $\tT(U)$ bijectively onto  $M \sm \{ 0
\}$. The bottom element of $\tCov(v)$ is given by the unique
tangible supervaluation  $\htv: R \to D(M)$ covering $v$ (cf.
\cite[Example 9.16]{IKR1}).

Returning to an arbitrary covering $\vrp: R \to U$ of $v$, we read
off from Lemma \ref{lem2.5} the~$\gm_*(\vrp)$ is tangible if
$\vrp$ is tangible. This implies

\begin{prop}\label{prop2.6} $\pusho{\gm}(\tCov(v)) \subset \tCov(\gm
v)$.

\end{prop}
We further have the following important fact.

\begin{thm}\label{prop2.7} The pushout of the initial covering
$\vrp_v : R \to U(v)$ of $v$ is the initial covering $\vrp_{\gm
v}: R \to U(\gm v)$ of $\gm v$. In particular $U(\gm v) = U(v)_
\gm$.

\end{thm}

\begin{proof}
Recall  that $\tT(U(v)) = R \setminus \gq$ and  $\tT(U(\gm v)) = R
\setminus \gq'$ with $\gq = \supp(v) $ and $\gq' = \supp (\gm v) =
v^{-1}(\gp)$. Thus it is fairly  obvious that $U(\gm v)=
U(v)_\gm$. If $a \in R$, we have
$$ \pusho{\gm}(\vrp_v)(a) \ = \ \al_{U,\gm}(\vrp_v(a));$$
hence, by Lemma \ref{lem2.5}, $ \pusho{\gm}(\vrp_v)(a) = \vrp(a)$
if $v(a) = e \vrp_v(a) \notin \gp$, while $ \pusho{\gm}(\vrp_v)(a)
= 0$ if $ v(a) \in \gp$. These are precisely the values attained
by $\vrp_{\gm v}$.
\end{proof}

We  focus on the restriction
$$ \tpusho{\gm}: \tCov(v) \to \tCov(\gm v)$$
of $\pusho{\gm}$ to tangible supervaluations. It maps the top
element $\vrp_v$ of $\tCov(\gm)$ to the top element $\vrp_{\gm v}$
of  $\tCov(\gm v)$. But it almost never  maps the bottom element
$\htv$ of $\tCov(v)$ to the bottom element $\widehat{\gm v}$ of
$\tCov(\gm v)$, as we will see below.

Our goal now is to exhibit a sublattice of $\tCov(v)$  which maps
bijectively onto $\tpusho{\gm}( \tCov(v))$ under the pushout map
$\tpusho{\gm}$. For that we need a construction of general
interest. \pSkip

In the following \emph{we always assume that $eU = M$ and $\tT(U)$
is closed under multiplication.} \pSkip

Given an ideal  $\ga$ of $M$ we introduce the equivalence relation
$$ \taE \ := \ E_{t,U}(\mfa) \ := \ \tE \cap E(M \setminus \ga), $$
with $\tE$ and $E(M \setminus \ga)$  the \MFCE-relations defined
in \cite[Examples 6.4.v and 6.12]{IKR1}. Clearly $\taE$ is a ghost
separating equivalence relation.

$E := \taE$ has the following explicit description: Let $x,y \in
U$ be given. If $x \in M$, or if $x \in \tT(U)$, but $ex \notin
\ga$, then $x \sim_E y$ iff $x =y$. If $x \in \tT(U)$  and $ex \in
\ga$, then $ x \sim_E y $ iff $y \in \tT(U)$ and $ex = ey$.

\begin{defn}\label{defn2.8} $ $

\begin{enumerate} \ealph
    \item We call the supertropical semiring $U / \taE$
    consisting of the $\taE$-equivalence classes the
    \textbf{t-collapse (= tangible collapse)} of $U$ over $\ga$
    and we denote this semiring by $\tac(U)$.

    \item We call the natural semiring homomorphism
$$ \pi _ {\taE} : U \to \tac(U)$$
the \textbf{t-collapsing map of $U$ over $\ga$}, and we denote
this map by $\tapi$, or $\taUpi$  if necessary.

    \item If $\vrp: R \to U $  is a tangible supervaluation
    covering $v$, we call the supervaluation
    $$ \vrp / \taE = \tapi \circ \vrp $$
    the \textbf{t-collapse of $\vrp$ over $\ga$}, and we denote
    this supervaluation by $\tac(\vrp)$.

    \item Finally, we say that $U$ is \textbf{t-collapsed over
    $\ga$}, if $\tapi$   is an isomorphism, for which we abusively
    write  $\tac(U) = U$, and we say that $\vrp$ is \textbf{t-collapsed over
    $\ga$} if $\tac(\vrp)) = \vrp$ (which happens iff $\tac(U) = U$,
    since our supervalutions  are assumed to be
    surjective).
\end{enumerate}
\end{defn}

We describe the semiring $\tac(U)$ more explicitly. Without
essential loss of generality we assume that $e\tT(U)_0 =M$.

If $\zz$ is any subset of $M$, let $U|_\zz$ denote the preimage of
$Z$ under the ghost map $\nu_U$,
$$ U|_\zz \ := \ \{ x \in U \ | \ ex \in \zz \}.$$

Now, if $U$ is t-collapsed over $\ga$, every $z \in U $ has a
unique tangible preimage under $\nu_U$. We denote this preimage by
$\chz$, and then have
$$ U|_\ga \ = \ \ga \ \dot \cup \  \chmfa$$
with $\chmfa = \{ \chz  \ds | z \in \ga\}$.

In general we identify
$$ \tac(U)|_{M \setminus \ga}  \ = \ U|_{M \setminus \ga} .$$
This makes sense since $[x]_{\taE} = \{ x \}$ for any $x \in U|_{M
\setminus \ga} . $ We then have
$$ \tac (U) \ = \ ( U|_{M \setminus \ga} \ \cup \ M ) \ \dot \cup \ \chmfa  $$
and
$$ U|_{M \setminus \ga} \ \cap \ M \ = \ M \setminus \ga.$$

After these identifications the following is obvious.
\begin{lem}\label{lem2.9} $ $
\begin{enumerate} \eroman
    \item If $x \in U$ then
    $$ \tapi(x) \ = \ \left\{
\begin{array}{lll}
  x &  & \text{if $x\in M$ or $ex \notin \ga$}, \\[2mm]
  \widecheck{(ex)} &  & \text{if } ex \in \ga\\
\end{array}
    \right.$$

    \item If $\vrp \in \tCov(U)$ and $a \in R$, then
$$ \tac(\vrp)(a) \ = \ \left\{
\begin{array}{lll}
  \vrp(a) &  & \text{if }  v(a) \notin \ga, \\[2mm]
  \widecheck{v(a)} &  & \text{if } v(a) \in \ga.\\
\end{array}
    \right.$$
\endbox
\end{enumerate}

\end{lem}

We now look at the map
$$ \tac: \tCov(v) \to \tCov(v)$$
which sends each $\vrp \in \tCov(v)$  to its t-collapse
$\tac(\vrp)$ over $\ga$. It is clearly order preserving, and is
idempotent, i.e., $(\tac)^2 = \tac$. We denote its image by
$\tcaCov(v)$. Its elements are the t-collapsed tangible
supervaluations over $\ga$  that cover $v$.

Using the description of suprema  and infima in the complete
lattice $\Cov(v)$ in \cite[\S7]{IKR1}, it is an easy matter to
verify the following
\begin{prop}\label{prop2.10}
$\tCov(v)$ is a complete sublattice of $\Cov(v)$, and $$\tac:
\tCov(v) \to \tCov(v)$$ respects suprema  and infima  in
$\tCov(v)$. Thus, also $\tcaCov(v)$ is a complete sublattice
of~$\Cov(v)$.
\end{prop}

\begin{rem}\label{rem2.11} Independently of this proposition it
is  clear that $\tcaCov(v)$ is a lower set in $\tCov(v)$  with top
element $\tac(\vrp _v)$. It follows that
$$  \tcaCov(v) \ = \ \{ \psi \in \Cov(v) \  | \ \tac(\vrp_v) \geq \psi \geq \chv \ \}.  $$
This proves again that $\tcaCov(v)$ is a complete sublattice of
$\Cov(v)$.
\end{rem}

We return to the surjective homomorphism $\gm: M \to N$ and now
choose for $\ga$ the prime ideal $\gp = \gm^{-1}(0)$ of  $M$.
\begin{prop}\label{prop2.12} Let $V := \tpc(U)$.
\begin{enumerate} \eroman
    \item The homomorphism $\tppi : U \to V$ induces an
    isomorphism $(\tppi)_\gm : U_\gm \tilde{\to}
    V_\gm$ over $N$. More precisely, using  the identifications from
    above we have $U_\gm = V_\gm$, and then $(\tppi)_\gm$  is the
    identity of $U_\gm$.
    \item $\al_{U, \gm} = \al_{V,\gm} \circ \tppi$.
    \item If $\vrp \in \tCov(v)$ then $\pusho{\gm}(\vrp) =
    \pusho{\gm}(\tpc(\vrp))$.
\end{enumerate}
\end{prop}

\begin{proof}
We have the identification
$$ \tT(U|_{M \setminus \gp}) \ = \ \tT(V|_{M \setminus \gp})$$
(see above). On the other hand, $\al_{U,\gm}$ maps $U|_\gp$ to $\{
0_N\}$, and $\al_{V,\gm}$ maps $V|_\gp$ to $\{ 0_N\}$. Finally
$$ \al_{U,\gm} | _M  \ = \ \al_{V,\gm} | _M  \ =  \gm.$$
Thus it is evident that, under our identifications, $U_\gm =
V_\gm$ and then $\al_{U,\gm} = \al_{V,\gm} \circ \tppi$. Reading
this equality as
$$ \id_{U_\gm} \circ \ \al_{U,\gm} \  =  \ \al_{V,\gm} \circ \ \tppi$$
we conclude by Proposition \ref{prop2.2}.a that $(\tppi)_\gm =
\id_{U_\gm}$. Finally, if $\vrp \in \tCov(v)$, then
$$ \pusho{\gm}(\tpc(\vrp)) \ = \ \al_{V,\gm} \circ \tpc(\vrp) \ = \
\al_{V,\gm} (\tppi (\vrp)) \ = \ \al_{U,\gm}( \vrp ) \ = \
\pusho{\gm}(\vrp).
$$
\end{proof}
\begin{lem}\label{lem2.13} Let $U$, $V$ be supertropical
semirings with $eU = eV =M$, and $\lm:U \to V$ a transmission over
$M$ with $\lm(\tT(U)) \subset \tT(V)$. Assume further that $U$ is
t-collapsed over~$\gp$. Finally assume that $\lm_\gm: U_\gm  \to
V_\gm$ is injective. Then $\lm: U \to V$ is injective.
\end{lem}

\begin{proof}
The upper square of the of the diagram in Proposition
\ref{prop2.2}.a restricts to a commuting square
$$\xymatrix{
           \tT(U | _{M \setminus \gp})  %
            \ar[d]_{``\lm"}
                \ar[rr]^{\underset{\widetilde{\phantom {w} \ }}{\phantom{u}}}_{\id}    &&
            \tT(U_\gm) \ar[d]_{``\lm_\gm"}
       \\
    \tT(V | _{M \setminus \gp})    \ar[rr]^{\underset{\widetilde{\phantom {w} \ }}{\phantom{u}}}_{\id}
     &&        \tT(V_\gm)
 }$$
Here the vertical arrows are restrictions of the maps $\lambda$
and $\lambda_\gamma$. The vertical arrow on the right is an
injective map by assumption. Thus, also the left vertical arrow is
an injective map. The restriction $\lm | \tT(U|_\gp)$ is a priori
forced to be injective, since $U$ is t-collapsed over~$\gp$.
Finally  $\lm$ restricts to the identity on $M$. Thus, $\lm$ is
injective.
\end{proof}

We now are ready for the main result of this section

\begin{thm}\label{thm2.14} As before assume that $\tT(U)$ is closed
under multiplication.
\begin{enumerate} \ealph
    \item The pushout map
    $$ \tpusho{\gm} : \tCov(v) \ \to \ \tCov(\gm v) $$
 restricts to a bijection from $\tcpCov(v)$ to $\pusho{\gm}(\tCov(\gm
 v))
 $. Consequently $\pusho{\gm}(\tCov(\gm v))$  is a sublattice of
 $\tCov(\gm v)$ isomorphic to $\tcpCov(v)$.

    \item If $\vrp, \psi \in \tCov(v)$  then $\pusho{\gm}(\vrp) =
    \pusho{\gm}(\psi)$ iff $\vrp$ and $\psi$ have the same
    t-collapse over $\gp$.
\end{enumerate}

\end{thm}

\begin{proof} a): Since we know already that $\pogm | \tcpCov(v) $
is a lattice homomorphism (Proposition \ref{prop2.10}), it
suffices to verify the following: If $\vrp, \psi \in \tCov(v) $
are t-collapsed over $\gp$ and $\vrp \geq \psi$, but $\vrp \neq
\psi$, then $\pogm(\vrp) \neq \pogm(\psi)$.

We have a unique surjective transmission $\lm:= \al_{\psi, \vrp}:
U \to V$ with $\psi =  \lm \vrp$. This implies  $\pogm(\psi) =
\lm_\gm \pogm(\vrp)$ by Corollary \ref{cor2.4}. If $\lm_\gm$ were
an isomorphism then also $\lm$ would be an isomorphism by Lemma
\ref{lem2.13} above. But this is not true. Thus $\lm_\gm$ is not
an isomorphism, and this means that $\pogm(\psi) \neq
\pogm(\vrp)$. \pSkip

b): We know by Proposition \ref{prop2.12} that $\pogm(\vrp) =
\pogm(\tpc(\psi))$. Thus $\pogm(\vrp) = \pogm(\psi)$ iff
$\pogm(\tpc(\vrp)) = \pogm(\tpc(\psi))$. By part a) this happens
iff $\tpc(\vrp) = \tpc(\psi)$.
\end{proof}

\medskip
We turn to the image of the map $\tpogm: \tCov(v) \to \tCov(\gm
v)$. Here we will put emphasis on strong supervaluations. Thus we
now assume in addition that the surjective valuation $v: R \to M$
is strong.

If $\vrp:R \to U$ is a strong supervaluation covering $v$, then
$\pogm(\vrp) = \al_{U,\gm} \circ \vrp$ is again a strong
supervaluation, as follows from \cite[Lemma 10.1.ii]{IKR1} and the
definition of ``strong'' \cite[Definition 9.9]{IKR1}. Thus
$$ \pogm(\tsCov(\vrp)) \ \subset \tsCov(\gm v).  $$
We have seen that $\pogm(\vrp_v) = \vrp_{\gm v}$, but we can only
state that the pushout $\pogm(\overline{\vrp}_v)$  of the initial
strong supervaluation $\overline{\vrp}: R \to \overline {U(v)}$ is
dominated by $\overline{\vrp}_{\gm v}: R \to \overline{U(\gm v)}$.
On the other side,  the pushout $\pogm(\htv)$ of the bottom
element $\htv : R \to D(M)$ of both $\tsCov(\vrp)$ and $\tCov(v)$
dominates $\widehat { \gm v} : R \to D(N)$. Using the
abbreviations
$$ \al := \al_{U(v), \gm}, \qquad \bar \al := \al_{\overline{U(v)},\gm}, \qquad \bt:= \al_{D(M), \gm},$$
we thus have a commuting diagram

$$
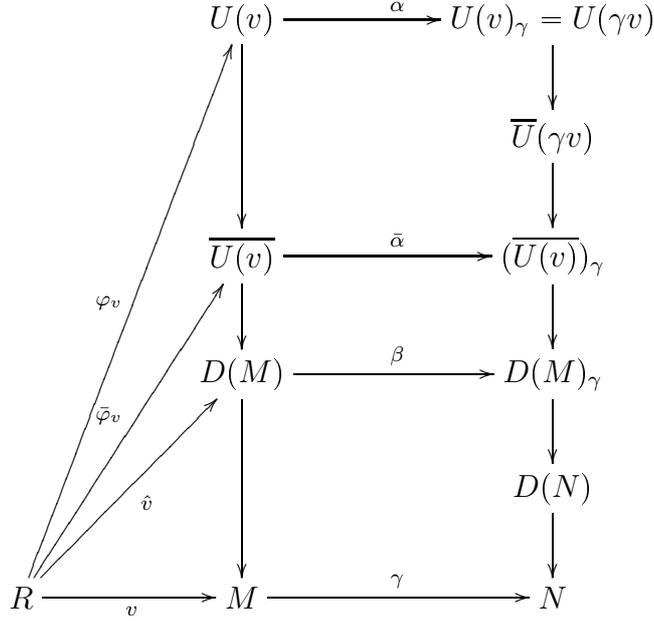
\xymatrix{
 &&     U(v) %
                \ar[rr]^{\al}   \ar[dd] &&
            U(v)_\gm = U(\gm v) \ar[d] \\
      &&  &&
            \overline U(\gm v)  \ar[d]
       \\
 &&     \overline{U(v)} %
                \ar[rr]^{\bar \al}   \ar[d] &&
            (\overline{U(v)})_\gm  \ar[d] \\
      &&     D(M) %
                \ar[rr]^{\bt}   \ar[dd] &&
            D(M)_\gm  \ar[d] \\
      &&  &&
            D(N)  \ar[d]
       \\
    R \ar[rr]_{v} \ar[rruu]_{\htv}
       \ar[rruuu]^{\bar \vrp_ v}
       \ar[rruuuuu]^{ \vrp_ v}
     & &  M   \ar[rr]^{\gm}   &&        N
 }$$
with surjective transmissions over $M$ and $N$ respectively as
vertical arrows.

The following questions immediately come to mind.
\begin{ques}\label{ques2.15} $ $  \begin{enumerate}
    \item Can we expect that $\overline{\vrp}_{\gm v } = \pogm (\overline{\vrp}_{v
    })$?

    \item Can we expect that $\widecheck {\gm v } = \pogm (\chv)$?

    \item Is $\pogm(\tCov(v))$  convex\footnote{A subset $Y$ of  a poset $X$ is called convex in $X$
    if $y \leq x \leq z$ for $y,z \in Y$, $x\in X$ implies that $x\in Y$.} in $\tCov(\gm v)$?
    \item Is $\pogm(\tsCov(v))$  convex  in $\tsCov(\gm v)$?
\end{enumerate}

\end{ques}
 Recall that $\tsCov(\gm v)$  is convex in $\tCov(\gm v)$, and $\tCov(\gm
 v)$ is convex in $\Cov(\gm v)$, as we have  seen in \cite[\S10]{IKR1}.

\medskip
Question (2) has a negative answer: If $z \in N \setminus \{ 0
\}$, then the tangible fiber of \\ $\{ x\in D(M)_\gm \ds | ex =
z\} $ is the union of the tangible fibers of $D(M)$ over the
points of $\gm^{-1}(z)$, and thus will quite often contain more
than one point. The other questions will be answered  here
completely only in a special case to which we turn now.

\medskip
Assume that $R \setminus \gq$  is a group under multiplication.
Then we can give a very explicit description of the map $\tpogm$,
and even $\pogm$.

\medskip
Now $M \smin \zset = v(R \smin \gq)$ and $N \smin \zset = \gm (M
\smin \zset)$ are groups, i.e., $M$ and $N$ are bipotent
semifields. This forces  $\gp = 0$ and $\gq = \gq'$.

Since $\gp =0$ we conclude from Theorem \ref{thm2.14} and
Proposition \ref{prop2.12} that $\pogm$ is an isomorphism of the
lattice $\tCov(v)$ onto its image $\pogm(\tCov(v))$, By
\cite[\S8]{IKR1} the \MFCE-relations on $U(v)$ except $E(\nu_U)$
are orbital, hence do not identify any tangibles with ghosts. Thus
$\Cov(v) = \tCov(v) \cup \{ v \}$ (as essentially observed in
\cite[\S8]{IKR1}). We have $\pogm(v) = \gm v$, and we conclude
that $\pogm$ \emph{is an isomorphism from $\Cov(v)$ onto its
image}.

\medskip
We have $M = \Gm \cup \zset$  with $\Gm$ an ordered abelian group.
Let $$\Dl : = \gm^{-1}(1_N).$$ This is a convex subgroup of $\Gm$,
since $\gm : M \to N$ is an order preserving monoid homomorphism.
The map $\gm$ induces an isomorphism from $M / \Dl = \Gm / \Dl
\cup \zset$ onto $N$. In the following we assume without loss of
generality that $N = M / \Dl$ and   $\gm$ is the map $x \mapsto
\Dl x$ from $M$ to $N$. Excluding a trivial case we assume that
$\Dl \neq 1$.

Returning to the notation from the end of \cite[\S10]{IKR1} we
have $$\go_v ^* = \{ a \in R \ | \ v(a) \in 1_M\} \quad \text{ and
} \quad \go_{\gm v} ^* = \{ a \in R \ | \ v(a) \in \Dl\} ,$$
further $\mm_v  = \{ a \in R \ | \ v(a) < 1_M\} $ and $\mm_{\gm v}
= \{ a \in R \ | \ v(a) < \Dl \} $. \{$v(a) < \Dl $ means $v(a ) <
\dl$ for every $\dl \in \Dl $.\}

If $H$ is a subgroup of $\go^*_v$ then $H$ is also a subgroup of
$\go_{\gm v}^ *$, since $\go^*_v$ is a subgroup of $\go_{\gm
v}^*$. Thus $H$ gives us a transmission
$$ \pi_{H, U(v)} : U(v) \ \to \ U(v) / E (H)$$
over $M$ and a transmission
$$ \pi_{H, U(\gm v)} : U(\gm v) \ \to \  U(\gm v) / E (H)$$
over $N$. \{Previously both maps sloppily had been denoted by
$\pi_H$.\}

\begin{thm}\label{thm2.16}
If $H$ is any subgroup of $\go^*_v$, then
\begin{enumerate} \ealph
    \item $(\pi_{H, U(v)})_\gm = \pi_{H, U(\gm v)}$,
    \item $\pogm(\vrp_v / H) = \vrp_{\gm v} / H$.
\end{enumerate}
\end{thm}

\begin{proof} a): Let $V := U(v) / E(H)$. We are done by
Proposition \ref{prop2.2}.a if we verify that
$$ \pi_{H, U(\gm v)} \circ \al_{U(v),\gm} \ = \ \al_{V, \gm} \circ \pi_{H,U(v)}.$$
This is easily  verified by use of Lemma \ref{lem2.5}. \pSkip

b): We know (Theorem  \ref{prop2.7}) that
$$ \vrp_{\gm v} \ = \
\pogm(\vrp_v) \ = \ \al_{U(v),\gm} \circ  \vrp_v$$ Thus
$$ \vrp_{\gm v} / H \ = \ \pi_{H, U(\gm v)} \circ  \al_{U(v),\gm} \circ \vrp_v.$$
On the other hand
$$ \pogm(\vrp_v / H) \ = \ \al_{V,\gm}(\vrp_v / H) \ = \ \al_{V,\gm} \circ \pi_{H,U(v)} \circ \vrp_v . $$
By step a) we conclude that indeed
$$\pogm(\vrp_v / H) = \vrp_{\gm v} / H.$$
\end{proof}

We learned before (\cite[\S8]{IKR1}) that the elements $\vrp$ of
$\tCov$ correspond uniquely with the subgroups $H$ of $\go_v^*$
via $\vrp = \vrp_v / H$, and now conclude by Theorem \ref{thm2.16}
that
$$ \pogm(\tCov(v)) \ = \ \{ \vrp_{\gm  v} / H \ | \ H \leq \go_v^*\}.$$
($``\leq "$ means subgroup). On the other hand
$$\tCov(\gm v) \ = \   \{ \vrp_{\gm  v} / H \ | \ H \leq \go_{\gm v}^*\}.$$
Thus, $\pogm(\tCov(v))$ is an upper set of the complete lattice
$\tCov(\gm v)$ with bottom element
$$ \pogm(\chv) \ = \ \vrp_{\gm v} / \go^*_v.$$
This element is definitely different from
$$ \widecheck{\gm v} \ = \ \vrp_{\gm v} / \go^*_{\gm v}, $$
since  $\go^*_{\gm v } / \go^*_v \cong \Dl$. Thus question
\ref{ques2.15}.(2) has a negative answer (which we know already),
while question \ref{ques2.15}.(3) has a positive answer.

How about question \ref{ques2.15}.(1)? The top element of
$\tsCov(v)$ is $\overline{\vrp}_v$. We saw in \cite[\S10]{IKR1}
that $\overline{\vrp}_v = \vrp_v / 1 + \mm_v$, and now conclude by
Theorem \ref{thm2.16} that
$$ \pogm(\overline{\vrp}_v) \ = \ \vrp_{\gm v} / (1 + \mm_{v}).$$
But
$$ \overline{\vrp}_{ \gm v}  \ = \ \vrp_{\gm v} / (1 + \mm_{\gm v}),$$
and $\mm_{\gm v}$ is definitely  smaller than $\mm_v$. Thus
$\overline{\vrp}_{ \gm v}  \gvertneqq \pogm(\overline{\vrp}_v)$.
Question \ref{ques2.15}.(1) has a negative answer.

Returning to the general situation, but still with $v : R \to M$
strong, we should expect that $ \overline{\vrp}_{ \gm v}
\gvertneqq  \pogm(\overline{\vrp}_v)  $ except in rather
pathological cases. Indeed, it seems often possible to pass from
$v: R \to M$ to a strong valuation $\tv  : \widetilde R  \to
\widetilde M$, with $\widetilde R$ a semifield by a localization
process (which we did not discuss), and to argue in $\tCov(\tilde
v)$.

Concerning applications, the strong supervaluations seem to be
more important than the others. But the fact that
$\pogm(\overline{\vrp}_v)$ differs from $\overline{\vrp}_{ \gm
v}$, while $\pogm({\vrp}_v) = {\vrp}_{\gm v}$, indicates that it
would not be advisable in  supervaluation theory   to restrict the
study  to strong supervaluations from the  start, as said already
in the Introduction.

\section {Supertropical predomains with prescribed ghost map}\label{sec:3}

For later use we give a generalization of  Construction 3.16 in
\cite{IKR1} of supertropical predomains. It merits independent
interest.

\begin{thm}\label{thm3.1} Assume that $M$ is a cancellative
bipotent semidomain. Assume further that $U=(U,\cdot \; )$ is an
abelian monoid, and $(M,\cdot \; )$ is a monoid ideal of $U$
(i.e., $M$ is a subsemigroup of $U$ and $UM\subset M).$ Assume
finally that a monoid homomorphism $p:U\to M$ is given (i.e., $p$
is multiplicative and $p(1_U)=1_M $) with $p(x)=x$ for every $x\in
M$ and $p^{-1}(0)=\{0\}.$ Then the following hold:
\begin{enumerate}
\item[i)] $0\cdot x=0$ for every $x\in U$, and $U\setminus \{0\}$
is closed  under multiplication. \item[ii)] On $U$ there exists a
unique addition $(+)$ extending the addition on $M$ such that
$(U,+,\cdot \; )$ is a supertropical semiring with $M$ the ghost
ideal and $p$ the ghost map of $U=(U,+,\cdot \; ).$ \item[iii)]
$U=(U,+,\cdot \; )$ is a supertropical predomain, and for
$x_1,x_2\in U$ we have the rule\footnote{Recall that every
bipotent semiring has a natural total ordering \cite[\S2]{IKR1}.}
\begin{equation}\label{3.1}
x_1+x_2=\begin{cases}x_1\   &\text{if}\quad  p(x_1)>p(x_2),\\
x_2\  &\text{if}\quad  p(x_1)<p(x_2),\\
p(x_1)\  &\text{if}\quad p(x_1)=p(x_2).
\end{cases}\end{equation}
\end{enumerate}
\end{thm}

\begin{proof}
We proceed in several steps.
\begin{enumerate}\ealph

\item  If $x\in U$, then $p(x\cdot 0)=p(x)p(0)=p(x)\cdot 0=0. $ Thus
$x\cdot 0=0.$ \pSkip

\item  If $x,y\in U\setminus\{0\}$, then $p(x)\ne0,$ $p(y)\ne0;$
hence $p(xy)=p(x)p(y)\ne0,$ and   $xy\ne0.$ Thus,
$U\setminus\{0\}$ is closed under multiplication. \pSkip

\item  We are forced to \textit{define} addition on $U$ by the
rule \eqref{3.1} above (cf. \cite[Theorem 3.11]{IKR1}). Clearly
this extends the given addition on $M.$ We have
$1_U+1_U=p(1_U)=1_M.$ \pSkip

\item  Write $1_M=e,$ $1_U=1.$ For $x\in U$ we have
$$e\cdot x=p(e\cdot x)=p(e)\cdot p(x)=e\cdot p(x)=p(x).$$
Thus, $p(x)=e\cdot x$ for every $x\in U.$ \pSkip

\item We start out to verify that $U$ is a semiring. Obviously,
the addition on $U$ is commutative, and it is easily checked that
the addition is also associative. For $x\in U$ we have $x+0=x$ if
$p(x)>0,$ and $x+0=0 $ if $p(x)=0$ iff $x=0.$ Thus, $0=0_M$ is the
neutral element of the addition on $U.$ \pSkip

\item It remains to  verify distributivity. Let $x_1,x_2,z\in U$
be given. If $x_1=0$ then $x_1z=0,$ $x_1+x_2=x_2;$ hence
$$x_1z+x_2z=0+x_2z=x_2z,$$ and thus
$$x_1z+x_2z=(x_1+x_2)z.$$
The same holds if $x_2=0$, and clearly also if $z=0.$

Assume now that $x_1,x_2,z\in G:=M\setminus\{0\}.$ If
$p(x_1)<p(x_2)$ then $p(x_1z)<p(x_2z)$ since $p(x_iz)=p(x_i)p(z)$
and the monoid $G$ is cancellative. Thus, $x_1+x_2=x_2,$
$x_1z+x_2z=x_2z$, and we see again that $$(x_1+x_2)z=x_1z+x_2z.$$ By
symmetry this also holds if $p(x_1)>p(x_2).$ In the case
$p(x_1)=p(x_2),$ we have $p(x_1z)=p(x_2z),$ $x_1+x_2=p(x_1),$ and
$$x_1z+x_2z=p(x_1z)=ex_1z=p(x_1)z=(x_1+x_2)z.$$ Now distributivity
is proved in all cases. \pSkip

\item We have proved that $U$ is a semiring with $x+x=ex=p(x)$ for
every $x\in U,$ and thus $M=p(U)=eU.$ The axioms (3.3$'$),
(3.3$''$), (3.4) from \cite[\S3]{IKR1} are now evident.  Thus, $U$
is supertropical and $\nu_U=p.$ The semiring $U$ is a
supertropical predomain.
\end{enumerate}
\end{proof}

\thmref{thm3.1} supersedes  Construction 3.16 in \cite{IKR1} since
here we do not need to assume that $U\setminus M$ is closed under
multiplication. Every supertropical semiring $U$ with $eU$ a
cancellative bipotent semidomain arises in the way indicated in
the theorem.

\begin{examp}\label{examp10.1} We discuss again the construction
of the supertropical semiring $U=U(v)$ for a valuation $v: R\to
M$, given in  \cite[Example 4.5]{IKR1}. Let $\mfq :=v^{-1}(0)$
  the support of~$v,$ and let $U$ denote the disjoint union of the
  sets $R\setminus\mfq $ and $M.$ We introduce on $U$ a
  multiplication $\odot$ as follows: For $x,y\in R\setminus
  \mfq $ and $z,w\in M$, put $$x\odot y=xy, \quad   x\odot z=z\odot
  x=v(x)z, \quad z\odot w=zw.$$ It is immediate that in this way $U$
  becomes an abelian monoid with $U\odot M\subset M.$ The map $p:
  U\to M$ given by $p(x)=v(x)$ for $x\in R\setminus\mfq ,$
  $p(z)=z$ for $z\in M$ is a monoid homomorphism and
  $p^{-1}(0)=\{0\}.$ \thmref{thm3.1} tells us that with the
  addition
  $$x\oplus y:=\begin{cases} x&\ \text{if}\quad p(x)>p(y)\\
  y&\ \text{if}\quad p(x)<p(y)\\
  p(x)&\ \text{if}\quad p(x)=p(y)\end{cases}$$ the monoid $U$
  becomes a supertropical semiring. The map $\varphi: R\to U$ with $\varphi(a)=a$
  for $a\in R\setminus\mfq ,$ $\varphi(a)=0$ for
  $a\in\mfq $ turns out to be a supervaluation covering $v.$
  \end{examp}

  \section{Transmissive equivalence relations}\label{sec:4}

  If a surjective transmission $\al : U\to V$ is given, $V$ can
  be identified with the set $U/E(\al )$ of equivalence classes
  of the equivalence relation $E(\al )$ \footnote{Recall that
  $E(\al )$ is defined by $x\sim_{E(\al )}y$ iff
  $\al (x)=\al (y).$} in such a way that
  $\al =\pi_{E(\al )}.$ We now pose the following problem: For which equivalence
  relations $E$ on a supertropical semiring $U$ can the set $U/E$
  be equipped with the structure of a (supertropical) semiring in
  such a way that $\pi_E: U\to U/E$ is a transmission?

  We first study the case $U=eU.$

  $U$ is a bipotent semiring, in other words, $U$ is a totally
  ordered monoid with absorbing smallest element $0$, cf. \cite[\S1]{IKR1}.

  Assume more generally that $M$ is a totally ordered set and $E$
  is an equivalence relation on $M.$ We want to install a total ordering on the set
  $M/E$  in such a way that the map $$\pi_E: M\to M/E, \qquad x\mapsto
  [x]_E,$$ is order preserving (in the weak sense; $x\le
  y\Rightarrow \pi_E(x)
\le \pi_E(y)).$ Thus we want that, if $\xi_1,\xi_2\in M/E$ and
$x_1\in\xi_1,$ $x_2\in\xi_2,$ then
$$x_1\le x_2\ \Rightarrow\ \xi_1\le \xi_2,$$
or, equivalently,
$$\xi_1>\xi_2\ \Rightarrow\ x_1>x_2.$$

It is clear  that such a total ordering on $M/E$ exists iff the
following holds. Given $\xi_1,\xi_2\in M/E$, either $x_1<x_2$ for
all $x_1\in\xi_1,$ $x_2\in \xi_2,$ or $x_1>x_2$ for all
$x_1\in\xi_1,$ $x_2\in \xi_2,$ or $\xi_1=\xi_2.$ More succinctly,
this condition can be written as follows:
\begin{align*}
(\text{OC}): \quad   & \text{If $x_1,x_2,x_3,x_4\in M,$} \text{
and
$x_1\le x_2,$ $x_3\le x_4,$ $x_1\sim_E x_4,$ $x_2\sim_E x_3,$} \\
& \text{ then $x_1\sim x_2$. }
\end{align*}
(Hence all $x_i$ are $E$-equivalent.)

 If an equivalence relation
$E$ on the totally ordered set $M$ obeys
 the rule (OC), we call $E$ \bfem{order compatible}.

It is sometimes useful to view order compatibility as a convexity
property. A subset $Y$ of $M$ is called \bfem{convex} (in $M$), if
for any $y_1,y_2\in Y$ and $x\in M$ with $y_1<x<y_2$, also $x\in
M.$

\begin{remark}\label{rem4.1}
An equivalence relation $E$ on the totally ordered set $M$ is
order compatible iff every equivalence class of $E$ is convex in
$M$.
\end{remark}

\begin{proof} a) If $y_1<x<y_2$ and $y_1\sim_E y_2,$ then (OC)
implies $y_1\sim_E x.$ (Take there $x_2 = x_3$.)\pSkip

 b) Assume that the equivalence
classes of $E$ are convex. We verify (OC). Let $x_1,x_2,x_3,x_4\in
M$ be given with $x_1\le x_2,$ $x_3\le x_4,$ and $x_1\sim_E x_4,$
$x_2\sim_E x_3.$

\textit{Case 1}. $x_2\le x_4.$ Now $x_1\le x_2\le x_4,$ and hence
$x_1\sim_E x_2.$

\textit{Case 2}. $x_2> x_4.$ Now $x_3\le x_4\le x_2,$ and hence
$x_4\sim_E x_2,$ and thus again $x_1\sim_E x_2.$\end{proof}

We present a proposition which is quite obvious from the initial
considerations on order compatibility  given above.

\begin{prop}\label{prop4.2} Let $M$ be a bipotent semiring and
$E$ an equivalence relation on the set~$M.$ There exists a
(unique) structure of a (bipotent) semiring on the set
\footnote{Recall that, for any set $Y\subset  U$ we write
$Y/E:=\{[y]_E\ds | y\in Y\}$.} $M/E$ such that the natural map
$\pi_E: M\to M/E,$ $x\mapsto [x]_E$ is a semiring homomorphism iff
$E$ is multiplicative and order compatible. In this case the
multiplication on $M/E$ is given by the rule $(x,y\in M)$
$$[x]_E\cdot [y]_E=[x\cdot y]_E,$$ and the ordering by the rule
$(\xi,\eta\in M/E)$
$$\xi\le \eta\quad\Leftrightarrow\quad \exists x\in\xi,\
y\in\eta\quad\text{with}\quad x\le y.$$
\end{prop}

\begin{proof}
Just notice that a map between bipotent semirings is a semiring
homomorphism iff it is multiplicative, sends 0 to 0, 1 to 1, and
is compatible with the orderings
(cf.~\cite[\S1]{IKR1}).\end{proof}

We  turn to equivalence relations on supertropical semirings
instead of just bipotent semirings.

\begin{defn}\label{defn4.4}
Let $U$ be a supertropical semiring. We call an equivalence
relation $E$ on~$U$ \bfem{transmissive} if on the set $U/E$ there
exists a semiring structure such that $U/E$ is supertropical and
the map $\pi_E: U\to U/E,$ $x\mapsto [x]_E$ is transmissive.
\end{defn}

We point out that, if $E$ is transmissive, the semiring structure
on $U/E$ is uniquely determined by the semiring structure of $U$
and the relation $E.$ This is clear from the following reasoning.

Assume a surjective transmission $\al : U\to V$ is given. Let
$E:=E(\al ).$ Since the map $\al $ is multiplicative, the
equivalence relation $E$ has to be multiplicative, and the
multiplication on $V$ is determined by $U$ and $\al $, since $\al
(x)\cdot\al (y)=\al (xy).$ We have $\al (e_U)=e_V,$ and $\al $
restricts to a surjective homomorphism $eU\to eV$ of bipotent
semirings. Thus, the restricted equivalence relation $E|eU$ is
order compatible, and the ordering on $eV$ is determined by the
ordering of $eU$ and the map $\al .$

It follows that the addition on $V$ is also determined by $U$ and
$\al $, since it can be expressed in terms of the multiplication
on $V,$ the element $e_V=\al (e_U)$ and the ordering of $eV$ (cf.
\cite[Theorem 3.11]{IKR1}).

Notice also that, if $x\in U$ and $ex\sim_E0,$ then $x\sim_E0$,
since $e\al (x)= \al (e  x)=0$ implies $\al (x)=0.$

We summarize these considerations as follows:

\begin{prop}\label{prop4.5}
Let $U$ be a supertropical semiring, $M:=eU,$ and assume that $E$
is a transmissive equivalence relation  on $U$. Then the following
is true: \begin{align*}& TE1: \  E
\quad\text{is   multiplicative}.\\
&TE2:\  \text{The equivalence relation $E|M$ is order
compatible}.\\ &TE3:\ \text{If}\quad x\in U  \quad\text{and}\quad
ex\sim_E0,\quad\text{then}\quad x\sim_E0.
\end{align*}
The structure of the supertropical semiring $U/E$ is uniquely
determined by the following data.

\begin{enumerate} \ealph
    \item[a)]  If $x,y\in U,$ then $[x]_E\cdot[y]_E=[xy]_E.$

\item[b)] The ghost ideal of $U/E$ is
$$M/E:=\{[x]_E\ds |x\in U\}.$$

\item[c)]  If $x,y\in M,$ then
$$x\le y\quad\Rightarrow\quad [x]_E\le[y]_E.$$
\end{enumerate}
\end{prop}

\begin{defn}\label{defn4.5}
We call an equivalence relation  on $U$ which has the properties
TE1-TE3 a \textbf{TE-relation}.
\end{defn}

Not every $TE$-relation is transmissive as will be clear from
\cite{IKRMon}.  Something ``non-universal" has to be added to
guarantee that a given $TE$-relation is transmissive. We now show
one such condition.

\begin{defn}\label{defn4.6}
We call a multiplicative equivalence relation $E$ on $U$
\bfem{ghost-cancellative} if the following holds.
 \begin{equation}
 \forall x,y,z\in eU: \ \text{If}\ xz\sim_Eyz,\quad\text{and}\quad
 z\nsim_E0,\quad\text{then}\quad x\sim_Ey.\tag{Canc}
\end{equation}

This means that the monoid $(M/E)\setminus\{0\}$ is cancellative.
$\{$If $U=M$, we usually say ``cancellative" for
``ghost-cancellative''.$\}$ \end{defn}

We arrive at the main result of this section.

\begin{thm}\label{thm4.7}
Let $U$ be a supertropical semiring and $M:=eU$ its ghost ideal.
Assume that $E$ is a TE-relation on $U$. Assume also that $E$ is
ghost-cancellative. Then $E$ is transmissive.\end{thm}

\begin{proof} Let $\olU  $ denote the set $U/E$, and, for any
$x\in U,$ let $\bar x =[x]_E.$ \propref{prop4.2} tells us that,
due to $TE1$ and $TE2$, we have the structure of a bipotent
semiring on the set
$$\olM :=M/E:=\{\bar x\ds | x\in M\},
$$
such that the map $M\to\olM $, $x\mapsto \bar x,$ is a semiring
homomorphism. It has the unit element $\bar e$ $(e:=e_U)$ and the
zero element $\overline 0.$ The assumption (Canc) means that $\olM
$ is cancellative. We have $\olU  \cdot\olM \subset \olM.$ The map
$p: \olU  \to \olM ,$ $p(\bar x):=\bar e\bar x=e\bar x$ is a
monoid homomorphism with $p(\bar x)=\bar x$ for $x\in M.$ The
assumption
 $TE3$ means that $p^{-1}(\overline 0)=\{\overline 0\}.$ Thus,
\thmref{thm3.1} applies and gives us the structure of a
supertropical semidomain on the set $\olU  $ with ghost map
$\nu_{\olU  }=p$ and ghost ideal $\olM .$

It remains to prove that the map $\pi_E: U\to \olU  ,$ $x\mapsto
\bar x,$ is a transmission. We have to check the axioms
$TM1$-$TM5$ in \cite[\S5]{IKR1}. The first four axioms $TM1$-$TM4$
are evident. $TM5$ holds, since indeed the map $M\to\olM $,
$x\mapsto \bar x,$ is a semiring homomorphism.
\end{proof}

This theorem allows a second approach to the key result of \S1,
Theorems \ref{thm1.12} and \ref{thm1.13}, which seems to be faster
than the route taken in \S1 (but perhaps gives less insight).

\begin{examp}\label{examp4.8}
We return to the assumptions of Theorems \ref{thm1.12} and
\ref{thm1.13}: $U$ is a supertropical semiring, and $\gm $ is a
surjective homomorphism from $M:=eU$ to a cancellative bipotent
semidomain $M'.$ We \bfem{define} a binary relation $F:=F(U,\gm )$
on $U,$ decreeing
\begin{align*} x_1\sim_F x_2\quad\Leftrightarrow\quad&\text{either}\quad
x_1=x_2,\quad\text{or}\quad \gm (ex_1)=\gm (ex_2),\ x_1=ex_1,\
x_2=ex_2,\\ &\text{or}\quad \gm (ex_1)=\gm (ex_2)=0.
\end{align*}


One verifies directly in an easy way that $F$ is an equivalence
relation. Clearly $F$ is multiplicative. The restriction
$$F|M:=F\cap (M\times M)$$ is order compatible, since $\gm $
preserves the ordering (in the weak sense). For $x\in U$ we have
$x\sim_F 0$ iff $\gm (ex)=0$ iff $ex\sim_E 0.$ Thus axioms
$TE1$--$TE3$ are valid. The semiring $M/F$ is isomorphic to $M'$
via $\gm ,$ and hence is a cancellative semidomain. Now
\thmref{thm4.7} tells us that the map $\pi_F$ is transmissive.

Then the proof of Theorem \ref{thm1.13} gives us that $\pi_F$ is a
pushout transmission. \{One does not need to know for this that
$\pi_F$ is initial.\} Alternatively, one may use a more general
result on pushout transmissions given below (\thmref{thm4.13}). In
particular, in Notation~\ref{notation1.7},
$$ F(U,\gm) = E(U,\gm).$$

\end{examp}

In \cite[\S8]{IKR1} we introduced \emph{orbital} equivalence
relations.  Typically a  relation $F(U,\gm )$, as just considered,
is almost never orbital.  We now ask for those orbital equivalence
relations which are transmissive.

\begin{lem}\label{lem4.9}
Let $M$ be a totally ordered set and $H$ an (abelian)
semigroup\footnote{All semigroups occurring in this paper are
assumed to be abelian.} which operates on $M$ in an order
preserving way. $\{$If $x,y\in M,$ $h\in H,$ and $x\le y$, then
$hx\le hy.\}$ We introduce on $M$ an equivalence relation
$E:=E(H)$ as follows:
$$x\sim_Ey\quad\Leftrightarrow\quad \exists g,h\in H: gx=hy.$$
Assume that for every $x\in M$ the orbit $Hx$ is convex in $M.$
Then $E$ is order compatible.
\end{lem}

\begin{proof}
We verify that every equivalence class of $E$ is convex, and then
will be done (cf.~Remark~\ref{rem4.1}). Let $x_1,x_2,y\in M$ be
given with $x_1<y<x_2,$ and $x_1\sim_Ex_2.$ There exist elements
$h_1,h_2$ in $H$ with $h_1x_1=h_2x_2.$ This implies
$$h_2x_1\le h_2y\le h_2x_2=h_1x_1.$$
Since $Hx_1$ is convex, there exists some $h_3\in H$ with
$h_2y=h_3x_1;$ hence $y\sim_E x_1.$\end{proof}

If $G$ is a (totally) ordered (abelian) cancellative semigroup, we
denote the group envelope of $G$ (given in the well-known way by
fractions $\frac{g_1}{g_2}$ with $g_1,g_2\in G)$ by $\langle
G\rangle.$ We equip~$\langle G\rangle $ with the unique ordering
which extends the given ordering of $G$ and is compatible with
multiplication.

\begin{thm}\label{thm4.10}
Let $U$ be a supertropical semiring with ghost ideal $M:=eU,$ and
let $H$ be a submonoid of $$S(U):=\{x\in U \ | \  x\mathcal
T(U)\subset \mathcal T(U)\}.$$ Finally, let
$$\mfq :=\{x\in M \ | \ \exists h\in H: hx=0\}=\{x\in
M \ | \ x\sim_H0\},$$ which is an ideal of $M.$ Assume that $M$ is
a semidomain.
\begin{enumerate}\item[a)] The semigroup $H$ operates on $M,$ and
hence on $M\setminus\mfq ,$ by multiplication in an order
preserving way. Either $\mfq $ is a lower set and a prime ideal of
$M,$ or $\mfq =M.$ \item[b)] If $\mfq \ne M,$ and the monoid
$M\setminus\mfq $ is cancellative, and the submonoid $\nu_U(H)=He$
of $M\setminus\mfq $ is convex in the ordered abelian group
$\langle M\setminus\mfq \rangle,$ then $E(H)$ is transmissive.
\item[c)] If $\mfq =M,$ then $U/E$ is the null ring, and hence
$E(H)$ is again transmissive.
\end{enumerate}
\end{thm}

\begin{proof} a) If $x_1,x_2\in M$,\ $h\in H,$ and $x_1\le x_2,$
then $x_1+x_2=x_2;$ hence $hx_1+hx_2=hx_2,$ and thus $hx_1\le
hx_2.$ If $x\in\mfq ,$ $y\in M$ and $y\le x,$ there exists some
$h\in H$ with $hx=0.$ We have $hy\le hx;$ hence $hy=0,$ and thus
$y\in\mfq .$ Thus $\mfq $ is a lower set of $M.$ Clearly,
$h(M\setminus\mfq )\subset M\setminus\mfq $ for every $h\in H.$

If $x,y\in M$ are given with $xy\in\mfq ,$ then there exists some
$h\in H$ with $hxy=0.$ Since $M$ is a semidomain, it follows that
$hx=0$ or $y=0,$ and hence $x\in\mfq $ or $y\in\mfq .$ This proves
that the ideal $\mfq $ of $M$ is prime. \pSkip

b) We will use \thmref{thm4.7}. The equivalence relation $E(H)$ is
multiplicative. For any $x\in U$ with $ex\sim_H0$, there exists
some $h\in H$ with $e(hx)=h(ex)=0.$ This implies $hx=0,$ and hence
$x\sim_H0.$ Thus $E(H) $ obeys $TE1$ and $TE3.$

We verify $TE2$ by proving that every equivalence class of
$E(H)|M$ is convex. Let $x_1,x_2,x_3\in M$ be given with
$x_1<x_2<x_3$ and $x_1\sim_H x_3.$ We need to be convinced that
$x_1\sim_H x_2.$

 \textit{Case} 1.  $x_1\in\mfq ,$ i.e., $x_1\sim_H0.$ Then
 $x_3\sim_H0.$ Since $\mfq $ is a lower set, we conclude
 that $x_2\sim_H0,$ and hence $x_1\sim_Hx_2.$

  \textit{Case} 2. $x_1\notin\mfq .$ Now all $x_i$ lie in
  $M\setminus \mfq ,$ since $M\setminus\mfq $ is an
  upper set. We verify that for every $x\in M\setminus\mfq $
  the orbit $Hx$ is convex in $M\setminus\mfq .$ Then
  \lemref{lem4.9} will tell us that the restriction of $E(H)$ to
  $M\setminus\mfq $ is order compatible. This will imply that
  $x_1\sim_Hx_2,$ as desired.

  Let $x,y\in M\setminus\mfq $ and $h_1,h_2\in H$ be given
  with $h_1x\le y\le h_2x.$ In the ordered abelian group $\langle
  M\setminus\mfq \rangle,$ we have $h_1\le yx^{-1}\le h_2.$
  By our convexity hypothesis, this implies $yx^{-1}=h_3\in H.$
  Thus $y=h_3x\in Hx,$ as desired.
  $TE2$ is verified.

  It remains to check that $E(H)$ is ghost-cancellative. Let
  $x,y,z\in M$ be given with $xz\sim_H yz,$ $z\nsim _H0.$ Thus
  $z\notin \mfq .$ We have elements $h_1,h_2$ in $H$ with
  $h_1xz=h_2yz.$

  If $x\in\mfq ,$ then $h_2yz\in\mfq ,$ and hence
  $y\in\mfq ,$ since $\mfq $ is prime. Thus $x\sim_Hy$
  in this case. The same holds if $y\in\mfq .$ Assume
  finally that $x,y\in M\setminus\mfq .$ The assumption that
  the monoid $M\setminus\mfq $ is cancellative implies that
  $h_1x=h_2y;$ hence, $x\sim_Hy$ again.

  Now \thmref{thm4.7} tells us that indeed $E(H)$ is
  transmissive.
  \pSkip

c) If $\mfq =M$ then $ex\sim_H0$ for every $x\in U,$ and hence
$x\sim_H0$ by an argument from (b) above. Thus $U/E(H)=\{0\}.$
  \end{proof}

  \begin{examp}\label{examp4.11} In the case that $U$ is a
  supertropical semifield, $M=\Gamma\cup\{0\}$ with $\Gamma$ an
  ordered abelian group, the situation addressed in
  \thmref{thm4.10} reads as follows:

  Let $H$ be a subgroup of $\mathcal T(U)$ whose image $\Delta:=He$
  in $\Gamma$ is convex in $\Gamma.$ Then $U/E(H)$ is just the
  orbit space $U/H$ (in the traditional sense), and $\mfq =\{0\}.$ We have
  $$\mathcal T(U/H)=\mathcal T(U)/H,\quad \mathcal
  G(U/H)=\Gamma/\Delta,\quad e_{U/H}=He.$$
  The map $\pi_H$ from $U$ to $U/E(H)$ sends an element $x$ of $U$
  to $Hx.$ It is a transmission. It covers the semiring
  homomorphism
  $$\gm _\Delta:\Gamma\cup\{0\}\to \Gamma/\Delta\cup\{0\},$$
  which sends an element $g$ of $\Gamma$ to $g\Delta$ and 0 to 0.

  If $\Delta\ne\{ e \},$ then $\pi_H$ is not a semiring
  homomorphism. Indeed, we can choose elements $x,y\in \mathcal
  T(U)$ with $Hx=Hy,$ but $ex<ey.$ Then $x+y=y;$ hence
  $\pi_H(x+y)=Hy,$ while $\pi_H(x)+\pi_H(y)=eHy=\Delta(ey).$
  Notice also that the transmission $\pi_H$ is not initial, since
  $E(H)$ is different from the relation $E(U,\gm _H)$ described
  in Example \ref{examp4.8}.\end{examp}

  We return to transmissive equivalence relations in general.

  \begin{defn}\label{defn4.12} We call a transmissive equivalence relation $E$ on a
  supertropical semiring $U$ \bfem{initial} (resp. \bfem{pushout})
  if the transmission $\pi_E: U\to U/E$ is initial (resp. pushout)
  (cf. Definitions \ref{defn1.2} and \ref{defn1.3}).
  \end{defn}

  We now bring a condition which guarantees that a given
  transmissive equivalence relation $E$ is pushout. The proof will follow
  essentially the same arguments as used in Theorem  \ref{thm1.13}  in
  the case considered there and reconsidered in Example
  \ref{examp4.8}.

  \begin{thm}\label{thm4.13}
  Assume that $E$ is a transmissive equivalence relation on a supertropical semiring
  $U$ with the following additional property:

\medskip
  If $x\in\mathcal
  T(U),$ $y\in U$, and $x\sim_Ey,$ then either $x=y$ or $x\sim_E0$
  (and hence $y\sim_E0).$
\medskip

\noindent  Then $E$ is pushout.

\end{thm}

\begin{proof} Let $M:=eU,$ and let $\gm _E: M\to M/E$ denote
the ghost component of the transmission $\pi_E:U\to U/E.$

In order to verify the pushout property of $\pi_E,$ assume that
$\delta: M/E\to N$ is a homomorphism from $M/E$ to a bipotent
semiring $N$ and $\bt : U\to V$ is a transmission covering
$\delta\circ\gm _E.$ \{In particular, $eV=N\}.$

We look for a transmission  $\eta: U/E\to V$ covering $\delta$
with $\eta\circ\pi_E=\bt .$
$$
\begin{xy}
\xymatrix{   U \ar@/^/@<+1ex>[rr]^\bt  \ar@{>}[r]_{\pi_E}  & U/E\ar@{.>}[r]_\eta & V  \\
 M \ar[u] \ar@{>}[r]_{\gm _E} & M/E
  \ar@{>}[r]_\delta \ar[u]
   &
 N \ar
 [u]
}
\end{xy}
$$
We are forced to define the map $\eta$ by the formula
$$\eta([x]_E)=\bt (x)\qquad (x\in U).$$
In order to prove that $\eta$ is a well-defined map, we have to
verify for $x,y\in U$ with $x\sim_Ey$ that $\bt (x)=\bt (y).$

 \textit{Case} 1. $x\in M,$ $y\in M.$ Now
 $$\bt (x)=\delta\gm _E(x)=\delta([x]_E)$$
 and  $\bt (y)=\delta([y]_E).$ Since $x\sim_Ey,$ we conclude
 that $\bt (x)=\bt (y).$ \pSkip

  \textit{Case} 2. $x\in\mathcal T(U).$ If $x=y,$ then, of course,
  $\bt (x)=\bt (y).$ Otherwise $x\sim_E0,$ $y\sim_E0$ by the
  hypothesis of the theorem; hence $ex\sim_E0,$ $ey\sim_E0.$ By
  the settled first case, we conclude that
  $e\bt (x)=\bt (ex)=0,$ which implies $\bt (x)=0.$ In the same
  way, $\bt (y)=0.$ Thus $\bt (x)=\bt (y)$ again.

  The case that $y\in\mathcal T(U)$ is now settled, too. Thus,
  $\eta$ is indeed a well-defined map. We have $\eta\pi_E=\bt .$

  Since both $\bt $ and $\pi_E$ are transmissions, and $\pi_E$ is
  surjective, we know by \cite[Proposition~6.1.ii]{IKR1} that $\eta$ is a
  transmission. By assumption $\bt (x)=\delta([x]_E)$ for every
  $x\in M.$ But also $\bt (x)=\eta([x]_E).$ Thus $\eta$ covers
  $\delta.$ The pushout property of $\pi_E$ is verified.
\end{proof}

\section{The equivalence relations $E(\mfa )$}\label{sec:5}

We study a class of transmissive equivalence relations which turns
out to be particularly well accessible.

If $R$ is a ring and $\mfa $ is an ideal of $R$ we have the
well-known equivalence relation  $``\!\!\!\mod \mfa $'' at our
disposal. We write down the obvious analogue of this relation for
semirings.

\begin{defn}\label{defn5.1} Let $R$ be a semiring and $\mfa$ an ideal of $R.$
We define an equivalence relation $E(\mfa)$ on $R$ as follows,
writing $\sim_{\mfa }$ instead of $\sim_{E(\mfa )}.$
$$x\sim_{\mfa }y\quad\Leftrightarrow \quad\exists a,b\in\mfa: x+a=y+b.$$\end{defn}

For $x\in R$ we denote the equivalence class $[x]_{E(\mfa)}$ more
briefly by $[x]_{\mfa },$ and denote the map $x\mapsto [x]_{\mfa
}$ from $R$ to the set $R/E(\mfa )$ usually by $\pi_{\mfa }$
instead of $\pi_{E(\mfa )}.$

If $x,y,z\in R$ and $x\sim_{\mfa } y,$ then clearly $x+z\sim_{\mfa
}y+z$ and $xz\sim_{\mfa }yz.$ Thus, we have a well-defined
addition and multiplication on the set $R/E(\mfa ),$ given by the
rules $(x,y\in R)$ \begin{align*}
[x]_{\mfa }+[y]_{\mfa }&:=[x+y]_{\mfa },\\
[x]_{\mfa }\cdot [y]_{\mfa }&:=[xy]_{\mfa }.\end{align*} With
these compositions $R/E(\mfa )$ is a semiring and $\pi_{\mfa}$ is
a homomorphism from $R$ onto $R/E(\mfa)$, cf.  \cite{qtheory}.

\begin{thm}\label{thm5.2}
If $R$ is supertropical, then for any ideal $\mfa $ of $R$ the
relation $E(\mfa )$ is transmissive.
\end{thm}

\begin{proof} Any homomorphism between supertropical semirings
clearly obeys the axioms TM1--TM5 from \cite[\S5]{IKR1}, hence is
a transmissive map. Thus our task is only to prove that the
semiring $U/E(\mfa)$ is supertropical.

We verify directly the axioms (3.3$'$), (3.3$"$), (3.3) from
\cite[\S3]{IKR1} for the semiring $U/E(\mfa),$ i.e.,
$$
\begin{array}{rl}
(3.3' )_{\mfa}: & \  1+1+1+1\sim_{\mfa}1+1,\\[1mm]
(3.3'' )_{\mfa}: & \  x+x \sim_{\mfa}y+y\Rightarrow
x+x\sim_{\mfa}x+y, \\[1mm]
(3.3)_{\mfa}: &  \  \pi_{\mfa}(x)\ne\pi_{\mfa}(y)\Rightarrow
\pi_{\mfa}(x)+\pi_{\mfa}(y)\in\{\pi_{\mfa}(x),\pi_{\mfa}(y)\}. \\
\end{array}$$

\noindent Clearly (3.3$')_{\mfa}$ holds since (3.3$'$) of
\cite{IKR1} holds for $R,$ and $(3.3)_{\mfa}$ holds since (3.3) of
\cite{IKR1} holds for $R$ and
$\pi_{\mfa}(x)+\pi_{\mfa}(y)=\pi_{\mfa}(x+y).$

We turn to $(3.3)_{\mfa}$. We are given $a,b\in \mfa$ with
$x+x+a=y+y+b.$ We add $c:=e(a+b)$ to both sides and obtain
$x+x+c=y+y+c.$ Since $c+c=c$ it follows that
$$(x+c)+(x+c)=(y+c)+(y+c).$$
Now (3.3$''$) for $R$ gives us
$$(x+c)+(x+c)=(x+c)+(y+c).$$
Thus $x+x\sim_{\mfa}x+y,$ as desired.
\end{proof}

Let again $R$ be any semiring. In contrast to the case of rings,
 different
ideals $\mfa ,\mfb $ of $R$ may give the same relation $E(\mfa
)=E(\mfb ),$ but this ambiguity can be tamed.

Clearly $\mfa _1:=[0]_{\mfa }$ is again an ideal of the semiring
$R$. It consists of all $x\in R$ with $x+a\in\mfa $ for some
$a\in\mfa .$ We call $\mfa _1$ the \bfem{saturum} of $\mfa,$ and
we write $\mfa _1=\sat\mfa .$ We call $\mfa $ \bfem{saturated} (in
$U$), if $\mfa =\mfa _1.$

\begin{prop}\label{prop5.3}
Let $R$ be any semiring and $\mfa ,\mfb $ ideals of $R.$
\begin{enumerate}\item[i)] $E(\mfa)=E(\sat\mfa )$; \item[ii)] $E(\mfa )\subset E(\mfb
)$ iff $\sat\mfa \subset \sat\mfb;$
\item[iii)] $\sat\mfa $ is the unique biggest ideal
$\mfa'$ of $R$ with $E(\mfa ')=E(\mfa ).$
\end{enumerate}
\end{prop}

\begin{proof}
a) If $\mfa \subset\mfb $ then $E(\mfa )\subset E(\mfb ).$
Conversely, if $E(\mfa )\subset E(\mfb ),$ then $[0]_{\mfa
}\subset [0]_{\mfb },$ i.e., $\sat\mfa \subset \sat\mfb .$ \pSkip

b) Let $\mfa _1:=\sat\mfa .$ If $x\sim_{\mfa_1}y$, then there
exist $z,w\in\mfa _1$ with $x+z=y+w,$ and there exist $a,b\in\mfa
$ with $z+a\in\mfa ,$ $w+b\in\mfa .$ It follows that
$$x+(z+a)+b=y+(w+b)+a,$$
which tells us that $x\sim_{\mfa }y.$ Thus $E(\mfa_1)=E(\mfa ).$
\pSkip

c) If $\sat\mfa \subset \sat\mfb$, then
$$E(\mfa )=E(\sat\mfa )\subset E(\sat\mfb)=E(\mfb ).$$ Now the claims i) and ii) are evident. \pSkip

d) If $E(\mfa ')=E(\mfa ),$ then it follows from ii) that
$\sat\mfa '=\sat\mfa ,$ and hence $ \mfa '\subset \sat\mfa .$
\end{proof}

Assume  now that $U$ is a supertropical semiring with ghost ideal
$M:=eU.$ Then we can give a very precise description of the
relation $E(\mfa )$ for any ideal $\mfa $ of $U.$

\begin{thm}\label{thm5.4} Let $\mfa $ be an ideal of $U.$
The  equivalence classes of the relation $E(\mfa )$ are the set
$[0]_{\mfa }=\sat\mfa $ and the one-point sets $\{x\}$ with $x\in
U\setminus \sat\mfa .$ More precisely the following holds:
\begin{enumerate}\item[i)] If $ex>e\mfa $ (i.e., $ex>ea\
\forall a\in\mfa ),$ then $[x]_{\mfa }=\{x\}.$
\item[ii)] If $ex\le ea$ for some $a\in\mfa ,$ then
$x\sim_{\mfa }0.$
\end{enumerate}
\end{thm}

\begin{proof}
i) Assume that $ex>e\mfa $ and $x\sim_{\mfa }y.$ There exist
elements $a,b$ in $\mfa $ with $x+a=y+b.$ Now $ex>ea,$ and hence
$x+a=x.$ From $ex=ey+ eb$ we conclude that $ex=\max(ey,eb).$ But
$ex>eb.$ Thus $ex=ey,$ and $y+b=y.$ We have $x=y.$ \pSkip

ii) If $ex<ea$ for some $a\in\mfa ,$ then $x+a=a,$ and hence
$x\sim_{\mfa }0.$ If $ex=ea$ for some $a\in\mfa ,$ then $x+a=ea,$
and hence again $x\sim_{\mfa }0.$\end{proof}

The set $e\mfa $ is an ideal of both $U$ and $M$; hence, it gives
us a relation $E_U(e\mfa )$ on $U$ and a relation $E_M(e\mfa)$ on
$M.$ It further gives us ideals $\sat_U(e\mfa)$ and $\sat_M(e\mfa
)$ of $U$ and $M$, respectively.

\begin{cor}\label{cor5.5}
Let $\mfa $ be an ideal of $U.$ \begin{enumerate} \item[i)]
 $\sat_U\mfa $ is the set of all $x\in U$ with $ex\le
c$ for some $c\in e\mfa .$ \item[ii)] $\mfa $ is saturated in $U$
iff $e\mfa $ is a lower set of $M$ and every $x\in U$ with
$ex\in\mfa $ is itself an element of $\mfa .$
\item[iii)] $\sat_U(e\mfa )=\sat_U(\mfa).$ \item[iv)] $\sat_M(e\mfa )=\sat_U(\mfa)
\cap M=e\sat_U(\mfa ).$ \item[v)] $E_U(\mfa)=E_U(e\mfa ).$
\item[vi)] The restriction $E_U(\mfa) | M=E_U(\mfa
)\cap(M\times M)$ of the relation $E_U(\mfa )$ to $M$ coincides
with $E_M(e\mfa ).$\end{enumerate}\end{cor}

\begin{proof}
(i) is evident from \thmref{thm5.4}, since $\sat_U(\mfa
)=[0]_{\mfa },$ and (ii), (iii) are evident from~(i). We then
obtain (iv) by applying (i) to both $U$ and $M.$ \{More generally,
$e\mfb =\mfb \cap M$ for any ideal $\mfb $ of $U.\}$ Claim (v) is
clear, because the description of $E_U(\mfa )$ does not change if
we replace $\mfa $ by $e\mfa .$ Finally, we read off (vi) by
applying \thmref{thm5.4} to both $U$ and~$M.$
\end{proof}

\begin{cor}\label{cor5.6} If $\mfa $ and $\mfb $ are
ideals of $U$, then $E_U(\mfa )\subset E_U(\mfb )$ or $E_U(\mfb
)\subset E_U(\mfa ).$
\end{cor}

\begin{proof} We may assume from the start that $\mfa $ and
$\mfb $ are saturated. Now $e\mfa $ and $e\mfb $ are lower sets of
$M.$ Thus, $e\mfa \subset e\mfb $ or $e\mfb \subset e\mfa .$ This
implies that $\mfa\subset \mfb $ or $\mfb \subset \mfa $ (cf.
Corollary~\ref{cor5.5}.i), hence $E(\mfa )\subset E(\mfb )$ or
$E(\mfb )\subset E(\mfa).$\end{proof}

\begin{examp}\label{examp5.7}
The unique maximal saturated proper ideal of $U$ is $$\mfa :
=\{x\in U \ds |ex<e\}.$$ It is easily seen to be a prime ideal
(provided $U$ is not the null ring), but perhaps $\mfa $ is not a
maximal ideal of $U.$ Take for example $U=M=\mathbb N_0=\mathbb
N\cup\{0\},$ where $\mathbb N$ is the ordered monoid
$\{1,2,3,\dots\}$ with standard multiplication and standard
ordering. Now $\mfa =\{0\},$ but $M\setminus\{1\}$ is the only
maximal ideal of $M.$\end{examp}

{}From Corollary \ref{cor5.5} we can read off further facts about
saturated ideals, which will be needed later on.

\begin{schol}\label{schol5.8} As before, $U$ is a supertropical
semiring, and $M:=eU.$
\begin{enumerate}\item[a)] An ideal of $U$ is saturated, iff
$e\mfa (=\mfa \cap M)$ is saturated in $M,$ and moreover every
$x\in U$ with $ex\in\mfa $ is an element of $\mfa .$
\item[b)] If $\mfc $ is a saturated ideal of $M,$ then
$\mfa :=\{x\in U \ds |ex\in\mfc \}$ is a saturated ideal of $U$
and $e\mfa =\mfc .$ \item[c)] The saturated ideals $\mfa $ of $U$
correspond uniquely with the ideals $\mfc $ of $M$ which are lower
sets via
$$\mfc =e\mfa \ (=\mfa \cap M),\qquad \mathfrak
a=\{x\in U \ds |ex\in\mfc \}.$$
\end{enumerate}\end{schol}

\begin{proof}
a) Clear from Corollary \ref{cor5.5}.ii,iv. \pSkip

b) We have $\mfc \subset \mfa ,$ and hence $\mfc =e\mfa .$ Now use
a). \pSkip

c) Now evident, taking into account Corollary
\ref{cor5.5}.ii.\end{proof}

The saturated ideals of $U$ form a chain (Corollary \ref{cor5.6}).
We ask: which of these ideals are prime ideals? In particular,
given a saturated ideal $\mfa \ne U,$ does there exist a saturated
prime ideal $\mfp \supset \mfa ?$ If ``Yes", which is the smallest
one?

These questions can be pushed to the ghost level by the following
simple observation.

\begin{lem}\label{lem5.9} Assume that $\mfa $ is an ideal
of $U$ with $e\notin \mfa .$ Then $\mfa $ is a prime ideal of $U,$
iff $e\mfa (=\mfa \cap M)$ is a prime ideal of $M$ and every $x\in
U$ with $ex\in e\mfa $ is an element of $\mfa .$\end{lem}

\begin{proof} a) If $\mfa $ is prime in $U,$ then
$e\mfa =\mfa \cap M$ is prime in $M.$ Moreover, if $x\in U$ and
$ex\in e\mfa ,$ then $ex\in\mfa .$ Since $e\notin\mfa ,$ it
follows that $x\in\mfa .$ \pSkip

b) Assume that $e\mfa $ is prime in $M$ and $x\in\mfa $ for every
$x\in U$ with $ex\in e\mfa .$ Let $y,z\in U$ be given with
$yz\in\mfa .$ Then $(ey)(ez)\in e\mfa ;$ hence, $ey\in \mfa $ or
$ez\in\mfa ,$ implying $y\in\mfa $ or $z\in\mfa .$ Thus $\mfa $ is
prime.\end{proof}

\textit{N.B.} The condition $e\notin \mfa $ is important here. For
example, if $\mathcal T(U)$ is closed under multiplication, then
$\mfa :=eU$ is prime in $U$, but $\mfa \cap M=M$ is not prime in
$M.$

\begin{prop}\label{prop5.10}
\quad{}

\begin{enumerate} \item[i)] The prime ideals $\mfa $ of $U$
with $e\notin\mfa $ correspond uniquely with the prime ideals
$\mfc $ of $M$ via $\mfc =e\mfa (=\mfa \cap M)$ and
$$\mfa =\{x\in U \ds |ex\in\mfc \}.$$

\item[ii)] $\mfa $ is a saturated prime ideal of $U$ iff
$e\mfa $ is a saturated prime ideal of $M.$ \end{enumerate}
\end{prop}

\begin{proof} i) is clear by \lemref{lem5.9}. Now ii) follows by
Scholium \ref{schol5.8}.a. (Notice that if $\mfa $ is a saturated
ideal of $U$ and $\mfa \ne U,$ then $e\notin\mfa ,$ since
$1+e=e.)$\end{proof}

\begin{thm}\label{thm5.11} Let $\mfa $ be a saturated ideal
of $U$ and $\mfa \ne U.$ Then
$$\mfb :=\{x\in U\ds |\exists n\in\mathbb N: ex^n\in\mfa\}$$
is a prime ideal of $U.$ It is the smallest prime ideal containing
$\mfa ,$ and it coincides with the radical $\sqrt{\mfa}$ of
$\mfa$, defined by $$\sqrt{\mfa} := \{ x \in U \ds | \exists n \in
\mathbb N : x^n \in \mfa\}.$$
\end{thm}

\begin{proof} a) If $\mfc $ is an ideal of $M,$ let
$$\sqrt{\mfc }:=\{x\in M \ds |\exists n\in\mathbb N:
x^n\in\mfc \}.$$ In this notation
$$\mfb =\{x\in U \ds |ex\in \sqrt{e \mfa} \}.$$
By \propref{prop5.10} it is clear that it suffices to prove that
$\sqrt{e\mfa }$ is the smallest saturated prime ideal of $M$
containing $e\mfa .$ We have $e \notin \mfa$, since otherwise the
relation $1 + e = e$  would imply that $1 \in \sat \mfa = \mfa$.
Thus $e \notin \mfc$, hence $e \notin \sqrt{\mfc }$.\pSkip

b) Let $\mfc :=e\mfa .$ This is a saturated ideal of $M,$ i.e., an
ideal of $M$ which is a lower set of $M$ (cf. Scholium
\ref{schol5.8}). Clearly $M\cdot \sqrt{\mathfrak
c}\subset\sqrt{\mfc },$  and hence $\sqrt{\mfc }$ is an ideal of
$M$. Let $x\in\sqrt{\mfc },$ $y\in M$ and $y<x.$ Choosing some
$n\in\mathbb N$ with $x^n\in\mfc ,$ we have $y^n\le x^n;$ hence,
$y^n\in\mfc ,$ and $y\in\sqrt{\mfc }.$ Thus $\sqrt{\mfc }$ is a
lower set of $M.$ The ideal $\sqrt{\mfc }$ is saturated in $M.$
\pSkip

c) Let $x,y\in M$ be given with $xy\in\sqrt{\mfc }.$ Assume that
$y\le x$. We have $y^2\le xy,$ and hence $y^2\in\sqrt{\mfc },$
implying $y\in\sqrt{\mfc }.$ This proves that $\sqrt{\mfc }$ is
prime in $M.$ \pSkip

d) Let $\mfp $ be a prime ideal of $M$ containing $\mfc.$ If
$x\in\sqrt{\mfc } $ then $x^n\in\mfc \subset \mfp $ for some
$n\in\mathbb N$, and hence $x\in \mfp .$ \pSkip

e) If $x \in U$ and $ex^n \in \mfa$ for some $n \in \N$, then $x^n
\in \mfa$ since $x^n + e x^n = e x^n$ and $\mfa$ is saturated.
Thus $\mfb= \sqrt{\mfa}$.
\end{proof}

Our proof that $E_U(\mfa )$ is transmissive (\propref{prop4.5})
does not rely on the criterion \thmref{thm4.7} (nor on any other
theory). In particular, it is not necessary to assume that
$E_U(\mfa ) $ is ghost-cancellative (i.e., the ghost ideal
$M/E_U(\mfa )$ of $U/E_U(\mfa )$ is cancellative, cf. \S2). In
fact, the following theorem tells us that this often does not
hold.

\begin{thm}\label{thm5.12} Assume that $M=eU$ is a cancellative
semidomain. Let $\mfa $ be a saturated ideal of $U$ with $\mfa \ne
U$. The following are equivalent:
\begin{enumerate}\item[(1)] The ghost ideal $M/E_U(\mfa )$
of $U/E_U(\mfa )$ is a cancellative semidomain. \item[(2)] $e\mfa
$ is a prime ideal of $M.$ \item[(3)] $ \mfa $ is a prime ideal of
$U.$\end{enumerate}\end{thm}

\begin{proof} a) We first study the case that $U$ is ghost, i.e.,
$U=M.$ Condition (1) means the following.
$$\forall x,y,x\in M: \quad xz\sim_{\mfa }yz, \  z\notin
\mfa \ \Rightarrow\ x\sim_{\mfa }y.$$
 If this holds, then taking $y=0$ we see that $\mfa $ is a
 prime ideal. This proves (1) $\Rightarrow$ (2).

 Assume now that $\mfa $ is prime. Let $x,y,z\in M$ be given
 with $xz\sim_{\mfa }yz$ and $z\notin \mfa .$

\textit{Case 1}. $x\in\mfa .$ Then $yz\in\mfa .$ Since $\mfa$ is
prime, we conclude that $y\in\mfa .$ Thus,
$x\sim_{\mfa}0\sim_{\mfa }y.$

\textit{Case 2}. $x\notin \mfa .$ Now $xz\notin \mfa .$ Taking
into account \thmref{thm5.4} we obtain $xz=yz.$ Since $M$ is
cancellative, this implies $x=y.$ Thus $x\sim_{\mfa }y$ in both
cases. This proves (2) $\Rightarrow$ (1). \pSkip

b) Let now $U$ be any supertropical semiring. The ideal $e\mfa $
is saturated in $M$ (cf. Scholium \ref{schol5.8}), and $M/E_U(\mfa
)=M/E_M(e\mfa )$ (cf. Corollary \ref{cor5.5}.vi).

Applying what has been proved to $M$ and $e\mfa $, we see that
$M/E_U(\mfa )$ is cancellative iff $e\mfa $ is prime in $M.$ By
\propref{prop5.10}.ii this is equivalent to $\mfa $ being prime in
$U.$\end{proof}

\begin{examp}\label{examp5.13} Let $M:=[0,1]$ be the closed unit
interval of $\mathbb R$ with the usual multiplication and the
addition $x+y:=\max(x,y).$ $M$ is a cancellative bipotent
semidomain. We choose some $\theta\in]0,1[.$ Then
$\mfa:=[0,\theta]$ is an ideal and a lower set of $M,$ and hence
is a saturated ideal of $M.$ But $\mfa $ is not prime, since the
half open interval  $]\theta,1]$ is not closed under
multiplication. In fact, the only saturated prime ideals of $M$
are $\{0\}$ and $[0,1[.$

The bipotent semiring $M/E(\mfa )$ can be identified with the
subset $\{0\}\cup]\theta,1]$ of $[0,1]$ equipped with  the new
multiplication
$$x\odot y=\begin{cases} xy &\text{if}\quad xy>\theta\\
0&\text{if}\quad xy\le\theta\end{cases}$$ and the addition
$$x\oplus y=\max(x,y).$$\end{examp}

\begin{thm}\label{thm5.14} If $\mfa $ is any ideal of a
supertropical semiring $U$, then the transmissive equivalence
relation $E(\mfa )$ is pushout (i.e., the transmission $\pi_{\mfa
}$ is pushout, cf.~Definition~\ref{defn4.12}).
\end{thm}

\begin{proof} We may assume that $\mfa $ is saturated.
Looking at the description of $E(\mfa )$ in \thmref{thm5.4}, we
realize that the hypothesis in \thmref{thm4.13} holds for
$E=E(\mfa ).$ Thus, $E(\mfa  )$ is pushout.\end{proof}

It follows that, in the terminology of Notation \ref{notation1.7},
$$E(\mfa )=E(U,\gm _{\mfa })$$
with $\gm _{\mfa }$ the map $x\mapsto [x]_{\mfa }$ from $M$ to
$M/E(\mfa )$ covered by $\pi_{\mfa }. $ Notice that
\thmref{thm5.14} is not covered by the central result Theorem
\ref{thm1.13} in \S1, since we do not assume cancellation for
$M/E(\mfa ).$

We draw a connection from the relations $E(\mfa )$ to other
equivalence relations.

\begin{thm}\label{thm5.15} Let $E$ be a $TE$-relation (e.g., $E$ is a
transmissive equivalence relation). The set $\mfq : =[0]_E$ is a
saturated ideal of $U$ with $E(\mfq )\subset E.$ Moreover, $\mfq $
is the biggest ideal $\mfa $ of $U$ with $E(\mfa )\subset
E.$\end{thm}

\begin{proof} a) If $x\sim_E0$, then $zx\sim_E0$ for any $z\in U.$
Thus $U\cdot \mfq \subset \mfq .$

b) From $e\mfq \subset\mfq $ we conclude that $e\mfq=\mfq \cap
M=[0]_E\cap M.$ This is convex in $M$ and contains $0$, hence is
a lower set of $M$. \pSkip

c) By axiom $TE3$ every $x\in U$ with $ex\in\mfq $ is an element
of $\mfq .$ \pSkip

d) Let $x,y\in\mfq $ be given, and assume without loss of
generality that $ex\le ey.$ Then $e(x+y)=ey\in\mfq ,$ and hence
$x+y\in\mfq .$ This completes the proof that $\mfq $ is an ideal
of $U.$ We conclude from c) and Scholium \ref{schol5.8}.a  that
this ideal is saturated. \pSkip

e) The equivalence classes of $E(\mfq )$ are  $\mfq =[0]_E$ and
one-point sets (\thmref{thm5.4}). Thus, certainly $E(\mfq)\subset
E.$ If $E(\mfa )\subset E$ then
$$\mfa \subset \sat\mfa =[0]_{\mfa}\subset [0]_E=\mfq .$$\end{proof}

\begin{notation}\label{notation5.16}
If $E,F$ are equivalence relations on a set $X$ with $E\subset F,$
we denote by $F/E$ the equivalence relation induced by $F$ on the
set $X/E.$ Thus, for $x,y\in X,$
$$[x]_E\sim_{F/E}[y]_E\quad\Leftrightarrow\quad x\sim_Fy.$$
\end{notation}

\begin{prop}\label{prop5.17} Let $E$ be a transmissive equivalence relation on $U$
and $\mfq :=[0]_E.$ We know by \thmref{thm5.15} that $\mfq$ is an
ideal of $U$ and $E(\mfq )\subset E.$ Let $$\olE:=E/E(\mfq ).$$
\begin{enumerate}\item[i)] $\olE $ is transmissive. \item[ii)]
$\olE $ is pushout iff $E$ is pushout.\end{enumerate}\end{prop}

\begin{proof} i) We have the factorization $\pi_E=\pi_{\olE}
\circ\pi_{\mfq }.$ Since $\pi_E$ and $\pi_{\mfq } $ are
transmissive and $\pi_{\mfq }$ is surjective, we conclude that
$\pi_{\olE }$ is transmissive (cf.~\cite[Proposition 6.1.ii or
Corollary~6.2]{IKR1}). \pSkip

 ii)
We have a natural commuting diagram of transmissions
$$
\begin{xy}
\xymatrix{   U   \ar @{>}[r]_-{\pi_ \mfq  }  &
U/E(\mfq)\ar@{>}[r]_
{\pi_{\olE }} & U/E  \\
 M \ar  [u] \ar@{>}[r]_ -{\gm _{\mfq }} & M/E(\mfq )
   \ar@{>}[r]_{\gm _  {\olE }}
   \ar [u] &
 M/E \ar  [u]
}
\end{xy}
$$
with $\gm _{\mfq }$ and $\gm _  {\olE }$ the ghost components of
$\pi_{\mfq }$ and $\gm _  {\olE }$, respectively. \thmref{thm4.13}
tells us that the left square is pushout in the category STROP.
Since $\gm _{\mfq }$ is surjective, it follows that the outer
rectangle is   pushout iff the right square is pushout (e.g.
\cite[p. 72, Exercise 8]{ML}). This gives the second claim.
\end{proof}

\section{Homomorphic equivalence relations}\label{sec:6}

Let $R$ be a semiring.

\begin{defn}\label{defn6.1} We call an equivalence relation
$E$ on $R$ \bfem{additive}, if
$$\forall x,y,z\in R: x\sim_Ey\Rightarrow x+z\sim_E y+z,$$
and \bfem{multiplicative}, if
$$\forall x,y,z\in R: x\sim_Ey\Rightarrow xz\sim_E yz.$$
We call $E$ \bfem{homomorphic}, if $E$ is both additive and
multiplicative.\end{defn}

If $E$ is homomorphic, we have a well-defined addition and
multiplication on the set $R/E,$ given by the rules ($x,y\in R):$
$$[x]_E+[y]_E=[x+y]_E,\quad [x]_E\cdot [y]_E= [xy]_E,$$ and these
make $R/E$ a semiring. Moreover, we can say that an equivalence
relation $E$ on $R$ is homomorphic, iff there exists a (unique)
semiring structure on the set $R/E$, such that $\pi_E: R\to R/E,$
$x\mapsto [x]_E,$ is a homomorphism.

In the following, $U$ is always a supertropical semiring and
$M:=eU$ is its ghost ideal.

\begin{examples}\label{examps6.2}
We have already seen two instances of homomorphic equivalence
relations on $U,$ namely, the MFCE-relations and the relations
$E(\mfa )$ with $\mfa $ an ideal of $U.$

On the other hand, if $\gm : M\to M'$ is a homomorphism from $M$
to a cancellative bipotent semiring $M',$ the transmissive
equivalence relation $E:=E(U,\gm )$ (cf. Theorem \ref{thm1.12})
will usually not be additive, hence not homomorphic. Indeed, if
$x_1,x_2\in M,$ $z\in \mathcal T(U)$ and $ex_1<ez<ex_2,$
$x_1\sim_E x_2,$ i.e., $\gm (x_1)=\gm (x_2),$ but $\gm (x_1)\ne
0,$ then $x_1+z=z\in\mathcal T(U)$ and $x_2+z=x_2\in M;$ hence,
$x_1+z\not\sim_Ex_2+z.$
\end{examples}

 We have the following remarkable fact, a special case of which
 occurred already in \thmref{thm5.2}.

 \begin{thm}\label{4.3a} Every homomorphic equivalence relation on
 $U$ is transmissive. \{In other terms, every homomorphic image of
 a supertropical semiring is again supertropical.\}\end{thm}

\begin{proof} As in the proof of \thmref{thm5.2}, we see that
the only problem is to prove that the semiring $U/E$ is
supertropical. For this only the axiom (3.3$''$) from
\cite[\S3]{IKR1} needs serious consideration.

Given $x,y\in U$ with $ex\sim_Eey$, we have to verify that
$ex\sim_Ex+y.$ We may assume that $ex\le ey$. Now, if $ex=ey,$
then $ex=x+y.$ If $ex<ey,$ then $x+y=y$ and $ex+y=y,$ hence
$$x + y = ex+y\sim_E ey+y=ey\sim_Eex.$$
Thus, indeed $ex\sim_Ex+y$ in both cases.\end{proof}

We seek a more detailed understanding of the homomorphic
equivalence relations on a supertropical semiring $U.$ As an
intermediate step we analyze the additive equivalence relations on
$U.$

\begin{prop}\label{prop6.3} Let $E$ be an equivalence relation on
$U.$ The following are equivalent.
\begin{enumerate}\item[(1)] $E$ is additive.
\item[(2)] $E$ obeys the following rules.
\begin{align*}
\AEi:  \quad &  x\sim_Ey\Rightarrow ex\sim_E ey. \\
\AEii: \quad & E|M \text{ is order compatible.}\\
\AEiii: \quad & \text{If $ex<ey$ and $ex\sim_Eey,$ then
$ex\sim_Ey.$}
\end{align*}

\end{enumerate}
\end{prop}

\begin{proof} We write $\sim$ for $\sim_E.$
$(1)\Rightarrow(2)$:

a) If $x\sim y,$ then
$$ex=x+x\sim y+x\sim  y+y=ey.$$ \pSkip

b) We verify that every equivalence class of $E|M$ is convex,
which will prove order compatibility of $E|M$ (cf. Remark
\ref{rem4.1}). Let $x_1,x_2,y\in M$ and assume that $x_1 \sim x_2$
and  $x_1<y<x_2.$ Then
$$y=x_1+y\sim x_2+y=x_2;$$ hence also $y\sim x_1.$ \pSkip

c) Assume that $ex<ey$ and $ex\sim_E ey.$ Then
$$y=ex+y\sim ey+y=ey\sim ex.$$

$(2) \Rightarrow (1):$ Given $x_1,x_2,z\in U$ with $x_1\sim x_2,$
we have to verify that $x_1+z\sim x_2+z.$ We may assume that $ex_1
\leq ex_2$.

We distinguish six cases.
\begin{enumerate}
    \item[1)] If $ez<ex_1,$ we have $z+x_i=x_i$ $(i=1,2).$

\item[2)] If $ez>ex_2,$ we have $z+x_i=z$ $(i=1,2).$

\item[3)] If $ex_1=ez<ex_2,$ then $z+x_1=ex_1,$ $z+x_2=x_2.$  By AE3, we
have $ex_1\sim x_2.$

\item[4)] If $ex_1<ez<ex_2,$ then $z+x_1=z,$ $z+x_2=x_2.$ By AE3, we have
$ex_1\sim z,$ $ex_1\sim x_2.$

\item[5)] If $ex_1<ez=ex_2,$ then $z+x_1=z,$ $z+x_2=ex_2.$ By AE3,
$ex_1\sim z.$ By AE1, $ex_2\sim ex_1.$

\item[6)] If $ex_1=ez=ex_2,$ then $z+x_1=ez$ and $z+x_2=ez.$
\end{enumerate}

We see that in all six cases indeed $z+x_1\sim z+x_2.$
\end{proof}

\begin{example}\label{examp6.4}
Assume that $E$ is fiber conserving, i.e., $x\sim_Ey$ implies
$ex=ey$ (\cite[Definition 6.3]{IKR1}. Then $E$ is additive.
Indeed, the conditions \emph{\text{AE1--AE3}} hold trivially,
\emph{\text{AE3}} being empty.
\end{example}

\begin{thm}\label{thm6.5}
Every additive equivalence relation $E$ on $U$ arises in the
following way. Choose a partition $(M_i\ds | i\in I)$ into
non-empty convex subsets of $M.$ Let $J$ denote the set of all
indices $i\in I$ such that $M_i$ has a smallest element $a_i$ and
$a_i\ne0.$ Choose for every $i\in  J$ an equivalence relation
$E_i$ on the fiber $\{x\in U\ds |ex=a_i\}$. If $x,y$ are elements
of $U$ with $ex\le ey$, define
$$
\begin{array}{lll}
x\sim_Ey: &  \Leftrightarrow& \text{There exists some
$i\in I$ with $ex,ey\in M_i$;  } \\
& &\text{and in  case $i\in J$, either $ex>a_i$} \\ &&  \text{or
$ex=a_i$ and $x\sim_{E_i}ex,$} \\ &&\text{or $ex=ey=a_i$ and
$x\sim_{E_i}y.$}
\end{array}$$

If $x,y\in U$ and $ex>ey$, define, of course, $x\sim_Ey:\
\Leftrightarrow y\sim_Ex.$
\end{thm}

\begin{proof}
Given an additive equivalence relation $E$ on $U$, this
description of $E$ holds with $(M_i\ds |i\in I)$ the set of
equivalence classes of $E\ds | M,$ indexed in some way, and
$E_i:=E\ds |M_i$ for $i\in J$, due to the properties AE1--AE3
stated in the \propref{prop6.3}. Conversely, if data $(M_i\ds
|i\in I)$ and $(E_i\ds |j\in J)$ are given, as indicated in the
theorem, it is fairly obvious that the binary relation defined
there is an equivalence relation obeying AE1--AE3. \{Notice that
the fiber $U_0$ over $0\in M$ is the one-point set $\{0\}$. Thus,
we may omit the index $i$ with $0\in M_i$ in the set $J.\}$
\propref{prop6.3} tells us that $E$ is additive.\end{proof}

When dealing with additive  equivalence relations, we now strive
for a more intrinsic notation than the one used in
\thmref{thm6.5}.

As noticed above (Remark \ref{rem4.1}), an additive (= order
compatible) equivalence relation $\Phi$ on $M$ is the same thing
as a partition of $M$ into convex subsets, namely, the partition
of $M$ into the equivalence classes of $\Phi,$
$$\Phi \ \widehat  =  \  (\xi \ | \ \xi\in M/\Phi).$$

\begin{notation}\label{notation6.6} $ $
\begin{enumerate}
    \item[a)] Given an additive equivalence relation  $\Phi$ on $M,$ define
$$L(\Phi):=\{x\in M\ds | x\ne0\ \text{and}\ x\le y\ \text{for
every}\ y\in M\ \text{with}\ x\sim_\Phi y\}.$$ Thus, $L(\Phi)$
consists of those $x\in L$ which are the smallest element of
$[x]_\Phi.$

\item[b)] If $E$ is an additive equivalence relation on $U,$ define
$$L(E):=L(E| M).$$
\end{enumerate}
\end{notation}

Of course, $L(\Phi)$ and $L(E)$ may be empty. Clearly,
$[0]_\Phi\cap L(\Phi)=\emptyset$ and $[0]_E\cap L(E)=\emptyset.$

We can rewrite \thmref{thm6.5} as follows:

\begin{theorem6.6'} \label{thm6.5'}
Given an additive equivalence relation $\Phi$ on $M$ and for every
$a\in L:=L(\Phi)$ an equivalence relation $E_a$ on the set
$$U_a:=\{x\in U\ds |ex=a\},$$ there exists a unique additive
equivalence relation $E$ on $U$ with $E|M=\Phi$ and $E|U_a=E_a$
for every $a\in L.$ It can be described as follows:

Let $x,y\in U$ and $ex\le ey.$

\begin{enumerate}\item[1)] If $ex\notin L,$ then
$$x\sim_Ey \ \Leftrightarrow \ ex\sim_{\Phi}ey.$$

\item[2)] If $ex=a\in L,$ but $ey>a,$ then
$$x\sim_Ey \ \Leftrightarrow \ ex\sim_\Phi ey \quad\text{and}\quad x\sim_{E_a}ex.$$

\item[3)] If $ex=ey=a\in L,$ then
$$x\sim_Ey\ \Leftrightarrow \ x\sim_{E_a}y.$$\end{enumerate}
\end{theorem6.6'}

\vskip 10pt

 We want to analyze   under which conditions on the
data $\Phi$ and $(E_a\ds | a\in L(\Phi))$  the additive relation
$E$  will also be multiplicative, hence homomorphic. For this we
need still another preparation, namely, a study of the set
$$A(E):=\{x\in U\ds |x\sim_E ex\}.$$

It turns out that it is appropriate to start with an even weaker
property of $E$ than additivity.

\begin{defn}\label{defn6.7} We call an equivalence  relation $E$
on the supertropical semiring $U$ \bfem{ghost compatible}, if the
condition \emph{\text{AE1}} from above holds, i.e.,
$$\forall x,y\in U: x\sim_Ey\Rightarrow ex\sim_Eey.$$
\end{defn}

Clearly, every multiplicative and every additive equivalence
relation is ghost compatible.

\begin{lem}\label{lem6.8}
{}\quad

\begin{enumerate}\item[a)] If $E$ is any equivalence relation on
 $U,$ then $M\subset A(E)$ and $A(E)+A(E)\subset A(E).$

\item[b)] If $E$ is ghost compatible, then
$$A(E)=\{x\in U \ds | \exists z\in M: x\sim_Ez\}.$$

\item[c)] If $E$ is multiplicative, then $A(E)$ is an ideal of
$U.$
\end{enumerate}
\end{lem}
\begin{proof}

a): It is trivial that $M\subset A(E).$ Let $x,y\in A(E)$ be given
with $ex\le ey$ (without loss of generality). If $ex<ey,$ then
$x+y=y\in A(E).$ If $ex=ey,$ then $x+y=ey\in M\subset A(E).$
\pSkip

b): Assume that $x\in U,$ $z\in M,$ and $x\sim_Ez.$ Then
$ex\sim_Eez=z,$ since $E$ is ghost compatible. It follows that
$x\sim_Eex.$ \pSkip

c): If $x\sim_E ex,$ then $zx\sim_Eezx$ for every $z\in U,$ since
$E$ is multiplicative. Thus $U\cdot A(E)\subset  A(E).$ It follows
by a) that $A(E)$ is an ideal of $U.$
\end{proof}

\begin{remark}\label{rem6.9} If $E$ is additive, then, using the data from Theorem 6.6$'$,
we have
$$A(E)=\{x\in U\ds | ex\notin L\}\cup \bigcup_{a\in L}\{x\in
U_a\ds | x\sim_{E_{a}}a\}.$$

\end{remark}

\begin{thm}\label{thm6.10}
Assume that $E$ is an additive  equivalence relation on $U$ with
the data $$\Phi:=E | M,\quad  \mfA:= A(E),\quad L:=L(\Phi), \quad
E_a:=E | U_a$$ for $a\in L.$ The following are equivalent:
\begin{enumerate} \item[a)] $E$ is multiplicative (hence homomorphic).
\item[b)] $\Phi$ is multiplicative. $\mfA$ is an ideal of
$U.$ For any $a\in L,$\ $x,y\in U\setminus \mfA$ with $ex=ey=a$,
and $z\in U$ with $za\in L:$ $$x\sim_{E_a}y  \ \Rightarrow \
zx\sim_{E_a}zy.$$
\end{enumerate}
\end{thm}

\begin{proof} a) $\Rightarrow$ b): evident.

b) $\Rightarrow$ a): Let $x,y,z\in U$ be given with $x\sim_Ey.$ We
have to verify that $xz\sim_Eyz.$ Since $E$ is ghost compatible
and $\Phi$ is multiplicative, $exz\sim_E eyz.$

\textit{Case} 1. \ $x\in\mfA  $ or $y\in\mfA  $. Due to
\lemref{lem6.8}.b, the set $\mfA  =A(E)$ is a union of equivalence
classes of $E.$ Thus both $x$ and $y$ are in $\mfA  .$ Since $\mfA
$ is assumed to be an ideal, $zx$ and $zy$ are in $\mfA  ,$ and
then
$$zx\sim_E ezx\sim_E ezy\sim_E zy.$$

\textit{Case} 2. \ $x\notin\mfA  $ and $y\notin\mfA  $. Now $ex\in
L$, $ ey\in L.$ Since $ex\sim_{\Phi} ey,$ it follows that $ex=ey=:
a\in L.$ Thus $x\sim_{E_a} y.$ If $za\notin L,$ then $zx$ and $zy$
are in $\mfA  ,$ and we conclude as above that $zx\sim_E zy.$ If
$za\in L,$ we conclude from $x\sim_{E_a}y$ by assumption b) that
$zx\sim_{E_{za}}zy.$

Thus, $zx\sim_Ezy$ in all cases.
\end{proof}

We introduce a special class of ghost-compatible equivalence
relations, and then will identify  the homomorphic relations among
these.

\begin{defn}\label{defn6.11} Let $\Phi$ be an equivalence
relation on the set $M,$ and let $\mfA  $ be a subset of $U$
containing $M.$ We define an equivalence relation $E:=E(U,\mfA
,\Phi)$ on $U$ as follows:
$$
\begin{array}{lll}
x_1\sim_E x_2  & \Leftrightarrow &  \text{Either}\ x_1=x_2,
\\ & &  \text{or} \ x_1\in\mfA  ,\ x_2\in\mfA,  \ \text{and} \ \
ex_1\sim_\Phi ex_2.\end{array}
$$ \hfill\quad \qed
\end{defn}

The equivalence classes of $E=E(U,\mfA  ,\Phi)$ are the sets
$\{x\in\mfA  \ds | ex\in\xi\},$ with $\xi$ running through
$M/\Phi,$ and the one point sets $\{x\}$ with $x\in
U\setminus\mfA  .$ Clearly $E$ is ghost compatible and $E|
M=\Phi.$

There is a structural characterization of $E(U,\mfA  ,\Phi).$

\begin{prop}\label{prop6.12} Let $E:=E(U,\mfA  ,\Phi)$ with
$\Phi$ an equivalence relation on $M$ and $\mfA  $ a subset of $U$
containing $M.$
\begin{enumerate}\item[i)] $A(E)=\mfA  .$
\item[ii)] $E$ is the finest ghost compatible equivalence relation
on $U$ with $E |M\supset\Phi$ and \\ $A(E)\supset\mfA  .$
\end{enumerate}
\end{prop}

\begin{proof} i): If $x\in\mfA  ,$ then clearly $x\sim_E ex.$
But, if $x\notin \mfA  ,$ then $x\ne ex,$ and hence
$x\not\sim_Eex.$ \pSkip

ii): Let $F$ be a ghost compatible equivalence relation on $U$
with $F|M\supset\Phi$ and $A(F)\supset\mfA  .$ Let $x\in U$ be
given. We verify that $[x]_E\subset [x]_F.$

\textit{Case} 1. \ $x\notin\mfA  .$ Now $[x]_E=\{x\}\subset
[x]_F.$

\textit{Case} 2. \ $x\in\mfA  .$ Let $y\sim_Ex.$ Then $y\in\mfA  $
and $ex\sim_\Phi ex.$ Thus, $x\sim_Fex,$ $y\sim_F ey,$ $ex\sim_F
ey.$ We conclude that $y\sim_Fx.$ Thus again $[x]_E\subset [x]_F.$
\end{proof}

\begin{thm}\label{thm6.13}  Let again $E:=E(U,\mfA  ,\Phi)$
with $\Phi$ an equivalence relation on $U$ and $\mfA  $ a subset
of $U$ containing $M.$
\begin{enumerate}\item[i)] $E$ is multiplicative iff $\Phi$ is
multiplicative and $\mfA  $ is an ideal of $U.$ \item[ii)] $E$ is
additive, iff $\Phi$ is order compatible and $\mfA  $ contains
every $x\in U$ with $ex\notin L(\Phi).$ \item[iii)] Thus, $E$ is
homomorphic, iff $\Phi$ is homomorphic and $\mfA  $ is an ideal
containing\linebreak $\nu_U^{-1}(M\setminus L(\Phi)).$
\end{enumerate}
\end{thm}

\begin{proof} a) We know that $$\mfA  =A(E)=\{x\in U\ds |
\exists z\in M: x\sim_Ez\},$$ and that $\mfA  +\mfA  \subset \mfA
.$ \pSkip

b) If $E$ is multiplicative, then, of course, $\Phi$ is
multiplicative, and $\mfA  $ is an ideal by \lemref{lem6.8}.c. If
$E$ is additive, then $\Phi$ is additive, which means that $\Phi$
is order compatible. Also then $\mfA  $ contains every $x\in U$
with $ex\notin L(\Phi)$ by Property AE3 in \propref{prop6.3}. If
$E$ is homomorphic, then all these properties hold. \pSkip

c) Assume now that $\Phi$ is multiplicative, and $\mfA  $ is an
ideal of $U.$ We want to prove that $E$ is multiplicative. Let
$x,y,z\in U$ be given with $x\sim_Ey.$ We want to verify that
$xz\sim_Eyz.$ If $x\in\mfA  ,$ then $y\in\mfA  $ and $ex\sim_\Phi
ey;$ hence, $exz\sim_\Phi eyz$. Since $xz,yz\in\mfA  $, we
conclude that $xz\sim_E yz.$ If $x\notin\mfA  , $ then $x=y$, and
hence $xz=yz.$ \pSkip

d) Assume that $\Phi$ is order compatible and $x\in\mfA  $ for
every $x\in U$ with $ex\notin L(\Phi).$ We want to prove that $E$
is additive, and we use the criterion of \propref{prop6.3} for
this. Clearly, $E$ obeys the axioms AE1 and AE2 there. It remains
to check AE3. Let $x,y\in U$ be given with $ex<ey$ and
$ex\sim_Eey,$ i.e., $ex\sim_\Phi ey.$ Then $ey\notin L(\Phi).$ By
our assumption on $\mfA  =A(E)$ it follows that $y\in\mfA  , $
i.e., $y\sim_E ey.$ We conclude that $ex\sim_Ey,$ as desired. Thus
$E$ is indeed additive. \pSkip

e) We have proved claims i) and ii) of the theorem. They implies
iii).
\end{proof}

We discuss the special case that $\Phi$ is the diagonal of $M$,
$\Phi=\diag M.$ In other words,\linebreak $x\sim_\Phi y$ iff
$x=y.$ We write more briefly $E(U,\mfA  )$ for $E(U,\mfA,\diag
M).$ Repeating Definition~\ref{defn6.11} in this case we have

\begin{defn}\label{defn6.14}
Let $\mfA  $ be any ideal of the supertropical semiring $U$
containing the ghost ideal $M$ of $U.$ The equivalence relation
$E:=E(U,\mfA  )$ on $U$ is defined as follows: Let $x,y\in U.$

\qquad If $x\notin \mfA  :\ x\sim_Ey\Leftrightarrow x=y.$

\qquad If $x\in \mfA  :\ x\sim_Ey\Leftrightarrow y\in\mfA  ,\
ex=ey.$ \hfill\quad \qed \end{defn}

Clearly $L(\diag M)=M\setminus\{0\}.$ Thus, \thmref{thm6.13} tells
us that the equivalence relation $E(U,\mfA  )$ is homomorphic.
This also follows from \cite[\S6]{IKR1}, since $E(U,\mfA  )$ is
obviously an MFCE-relation.

Thus, the set $U/E$ with $E:=E(U,\mfA  )$ is a supertropical
semiring, the addition and multiplication being given by $(x,y\in
U):$
$$[x]_E+[y]_E:=[x+y]_E,\quad [x]_E\cdot[y]_E:=[xy]_E.$$

Every equivalence class $[x]_E$ of $E$ contains a unique element
of the set $$V:=(U\setminus\mfA  )\cup M,$$ namely, the element
$x$, for $x\notin \mfA  ,$ and the element $ex,$ for $x\in\mfA  .$
Notice that $V$ is closed under addition (Remark \ref{rem6.9}.b).

Identifying the set $U/E$ of equivalence classes of $E$ with the
set  of representatives $V$, we arrive at the following theorem.

\begin{thm}\label{thm6.15} Let $\mfA  $ be an ideal of $U$
containing $M$ and $V:=(U\setminus\mfA  )\cup M.$ On $V$ we define
an addition + and multiplication $\odot$ as follows:
$$x+y\ \text{is the sum of}\ x\ \text{and}\ y\ \text{in}\ U.$$
$$x\odot y:=\begin{cases}xy & \ \text{if}\quad {xy}\notin\mfA  ,\\
exy & \ \text{if}\quad xy\in\mfA  .\end{cases}$$ Then
$V=(V,+,\odot)$ is a supertropical semiring, and the map $\al :
U\to V$ with $\al (x)=x$ for $x\in U\setminus\mfA  ,$ $\al (x)=ex$
for $x\in\mfA  $ is a surjective semiring homomorphism. It gives
the equivalence relation $E(\al )=E(U,\mfA  ).$
\end{thm}

Of course, this can also be verified in a direct straightforward
way.

\begin{remarks}\label{rem6.16} $ $
\begin{enumerate} \eroman
    \item The sub-semiring $M$ of $U$ is also a sub-semiring of $V$ (in its
given semiring structure). In particular, $e_U=e_V.$

    \item
 $M$ is also the ghost ideal of $V,$ and the ghost map $\nu_V$ is
the restriction of $\nu_U$ to $V.$

\item We have $1_U=1_V$ if $1_U\notin \mfA  ,$ and $1_M=1_V$
if $1_U\in\mfA  .$ In the latter case $V=M.$
\end{enumerate}

\end{remarks}

\begin{example}\label{examp6.17}
Let $L$ be a subset of $M$ with $M\setminus L$ an ideal of $M.$
Define
$$\mfA  =\mfA  _L:=\{x\in U\ds | ex\in M\setminus
L\}\cup M=\nu_U^{-1}(M\setminus L)\cup L.$$ Then $\mfA  $ is an
ideal of $U$ containing $M.$ It is easily checked that $E(U,\mfA
)$ is the equivalence relation on $U$ which we considered in
\cite[Example  6.13]{IKR1}. We have $$V=(M\setminus
L)\cup\nu_U^{-1}(L).$$ If $L\cdot L\subset L,$ then $V\cdot
V\subset V;$ hence, the supertropical semiring $V$ is a
sub-semiring of $U,$. This is the case considered in \cite[Example
6.12]{IKR1}.

\end{example}

\begin{defn}\label{defn6.18}
We call an equivalence relation $E$ on $U$ \bfem{strictly ghost
separating} if no $x\in\mathcal T(U)$ is $E$-equivalent to an
element $y$ of $M.$ Under the very mild assumption that~$E$ is
ghost compatible, this means that $A(E)=M$ (cf.
\lemref{lem6.8}.b).\footnote{We reserve the label ``ghost
separating'' for a slightly broader class of equivalence relations
to be introduced in \cite{IKRMon}.}
\end{defn}

The restriction of $E(U,\mfA  )$ to the supertropical semiring
$V=(U\setminus\mfA  )\cup M$ from above is always ghost
separating. Moreover, we have the following facts.

\begin{prop}\label{prop6.19}
Assume that $F$ is a multiplicative equivalence relation (and
hence~$A(F)$ is an ideal of $U$), and $\mfA  $ is an ideal of $U$
with $M\subset \mfA  \subset A(F).$
\begin{enumerate}\item[i)] $E(U,\mfA  )\subset F.$
\item[ii)] The equivalence relation $\olF : =F/E(U,\mfA )$ on $\olU
 :=U/E(U,\mfA  )$ is again multiplicative, and
$A(\olF )$ is the image of $A(F)$ in $\olU  ,$ i.e., $A(\olF
)=A(F)/E(U,\mfA  ).$ \item[iii)] $\olF $ is strictly ghost
separating iff $\mfA  =A(F).$
\item[iv)] If we identify $\olU  $ with the semiring
$V:=(U\setminus \mfA  )\cup M,$ as explicated above, then $\olF
=F|V.$ \item[v)] $\olF $ is transmissive iff $F$ is transmissive.
\item[vi)] $\olF $ is homomorphic iff $F$ is
homomorphic.\end{enumerate}
\end{prop}

\begin{proof} Let $E:=E(U, \mfA  ).$

a) We claim that for any $x,y\in U$ with $x\sim_Ey$ also
$x\sim_Fy.$ Now, if $x\notin \mfA  ,$ then $x=y.$ If $x\in\mfA  ,$
then $y\in\mfA  $ and $ex=ey.$ Since $\mfA  \subset A(F),$ it
follows that $x\sim_F ex,$ $y\sim_F ey,$ and then that $x\sim_Fy.$
Thus $x\sim_Fy$ in both cases. This proves  $E\subset F.$ \pSkip

b) Claims ii) -- iv) of the proposition are fairly obvious. v)
follows from \cite[Corollary~6.2]{IKR1} since
$\pi_F=\pi_{\olF}\circ \pi_E$,  and $\pi_E$ is a surjective
homomorphism. vi) is again obvious.
\end{proof}

We now exhibit a case where we have met the equivalence relation
$E(U,\mfA ,\Phi)$ before. First a very general observation.

\begin{remark}\label{rmk:6.21} Every $\mfa$ of $U$  with $e \cdot \mfa  \subset \mfa$ is closed under
addition. The reason is, that for any $x,y \in U$ the sum $x+y$ is
either  $x$ or $y$ or $ex$. Thus every subset $\mfa$ of $U$ with
$U \cdot \mfa \subset \mfa$ (i.e., $\mfa$ a monoid ideal of $U$)
is an ideal of $U$. If $\mfa$ and  $\mfb$ are ideals of $U$ then
$\mfa \cup \mfb = \mfa + \mfb$.
\end{remark}

Assume that $\Phi$ is a homomorphic equivalence relation on $M$.
It gives us the homomorphism $\pi_\Phi$ from $M$ to the bipotent
semiring $M / \Phi$. We define
$$ \mfa_\Phi := \{ x \in U \ | \  ex \sim_\Phi 0\} $$
which is an ideal on $U$, and define
$$ \mfA := M \cup \mfa_\Phi = M + \mfa_\Phi,$$
which is an ideal of $U$ containing $M$. It is the set of all
$x\in U$ with $x =ex$ or $ex \sim_\Phi 0$. If necessary we more
precisely write $\mfa_{U,\Phi}$, $\mfA_{U,\Phi}$ instead of
$\mfa_{\Phi}$, $\mfA_{\Phi}$. Starting from
Definition~\ref{defn6.11} it can be checked in a straightforward
way that the multiplicative equivalence relation
$$ E:= E(U, \mfA_{\Phi}, \Phi)$$
has the following description ($x,y \in U$):
$$\begin{array}{lll}
    x \sim_E y & \Leftrightarrow & \text{either } x =y \\
&& \text{or } x= ex, \  y= ey, \  ex \sim_\Phi ey\\
&& \text{or } ex \sim_\Phi ey  \sim_\Phi 0.\\
  \end{array}
$$
Thus $E$ is the equivalence relation $F(U,\gm)$ defined in Example
\ref{examp4.8} with $\gm:= \pi_\Phi$, If $M / \Phi$ is
cancellative the we know from Theorem \ref{thm1.12} and Example
\ref{examp4.8} that $E(U, \mfA_{\Phi}, \Phi)$ is transmissive.
There are other cases where this also holds, cf. Remark
\ref{rmk:6.23} below.

We now apply Proposition \ref{prop6.19} to the relation $$F:= E(U,
\mfA \cup \mfa_{\Phi}, \Phi)$$ for $\mfA$ any ideal of $U$
containing $M$. Let $\olU$ denote the supertropical semiring $U/
E(U, \mfA)$, whose ghost ideal has been identified above with $M =
eU$. It again can be checked in a straightforward way that the
equivalence relation $F/ E(U, \mfA)$ on $\olU$ is just the
relation
$$ E(\olU, \mfA_{\olU,\Phi},\Phi) = F(\olU, \pi_\Phi),$$ in the
notation of Example \ref{examp4.8}. Thus we arrive at the
following result.
\begin{thm}\label{thm:6.23} Let $\Phi$ be a homomorphic equivalence
relation on $M := eU$ and $\mfA$ an ideal of $U$ which contains
$M$. Let $\mfD := \mfA \cup \mfa_\Phi = \mfA + \mfa_\Phi$ with
$\mfa_\Phi := \{ x\in U | ex \sim_\Phi 0\}$.
\begin{enumerate} \ealph
    \item Then $\olU := U/ E(U, \mfA)$ is a supertropical semiring
    (as we know for long) and  \\ $E(U, \mfD, \Phi) / E(U, \mfA)$ is
    the multiplicative equivalence relation $F(\olU,\pi_\Phi)$.

    \item $ E(U, \mfD, \Phi)$ is transmissive iff
    $F(\olU,\pi_\Phi)$ is transmissive.

    \item Tn particular $ E(U, \mfD, \Phi)$ is transmissive if $M/
    \Phi$ is cancellative.
\end{enumerate}
\end{thm}

\begin{remark}\label{rmk:6.23} Looking at Theorem \ref{thm6.13}
and Proposition \ref{prop6.19}.vi we can also state the following:
$ F(U, \mfD, \Phi)$  is homomorphic  iff $F(\olU,\pi_\Phi)$ is
homomorphic iff $\nu^{-1}_U(M\ L(\Phi)) \subset \mfD$.

\end{remark}

\begin{remark}\label{rmk:6.24}
The question might arise whether the $ E(U, \mfA, \Phi)$ is
transmissive for \textbf{any} ideal $\mfA \supset M$  of $U$ if,
say,  $M/ \Phi$ is cancellative. The answer in general is ``No'':
 If $E(U, \mfA, \Phi)$ is transmissive then $\mfA$ must contain
the ideal $\mfa_\Phi$. The reason is that for any transmission
$\al: U \to V$ and $x\in U $ with $\al(ex)=0$ we have $\al(x)=0$
since $\al(ex) = e\al(x)$.
\end{remark}


\end{document}